\documentclass[a4paper,twoside,10pt,leqno]{article}
\usepackage{amsmath,amscd,amsthm,amssymb,enumerate, url, eucal}

\newtheoremstyle{nodot}%name of style
{3pt}%Space above
{3pt}%Space below
{\itshape}%Body font
{}%Heading indent
{\bfseries} %Heading font
{}%Punctuation after heading
{0pt}%Space after head, 0pt is normal interword space
{\thmnote{#3}}%Typist note if present

\swapnumbers
\theoremstyle{plain}
\newtheorem{theorem}{Theorem}[section]

\newtheorem{corollary}[theorem]{Corollary}
\newtheorem{proposition}[theorem]{Proposition}

\theoremstyle{definition}

\newtheorem{definitions}[theorem]{Definitions}

\newtheorem{example}[theorem]{Example}

\newtheorem{notation}[theorem]{Notation}

\newtheorem{remarks}[theorem]{Remarks}

\newtheorem{hypotheses}[theorem]{Hypotheses}

\newcommand{\gen}[1]{\langle#1\rangle}
\newcommand{\series}[1]{\langle\langle#1\rangle\rangle}

\def \naturals{\mathbb{N}}
\def \integers {\mathbb{Z}}
\def \rationals{\mathbb{Q}}

\def\d1{\discretionary{-}{}{-}}
\def\coloneq{\mathrel{\mathop\mathchar"303A}\mkern-1.2mu=}

\renewcommand{\le}{\leqslant}
\renewcommand{\ge}{\geqslant}

\begin{document}

\pagestyle{myheadings}
\markboth{Ring coproducts embedded in power-series rings}{Pere Ara and Warren Dicks}
\title{Ring coproducts embedded in power-series rings}

\author{Pere Ara  and Warren Dicks}

\date{\today}

\maketitle

\begin{abstract} Let $R$ be a ring (associative, with $1$),
and let $R\series{a,b}$ denote the power\d1series $R$-ring in two
non-commuting, $R$-centralizing variables, $a$ and $b$.
Let $A$~be an $R$-subring of $ R\series{a}$ and $B$ be an
$R$-subring of $R\series{b}$, and let $\alpha$~denote the natural
map \mbox{$A \amalg_R B \to R\series{a,b}$}.
This article  describes some situations where $\alpha$ is injective and some where it is not.

We prove that if $A$ is a right Ore localization of $R[a]$ and $B$
is a right Ore localization of $R[b]$, then $\alpha$ is injective.
For example, the group ring over $R$ of the free group on
$\{1{+}\,a,1{+}\,b\}$ is
\mbox{$R[ (1{+}\,a)^{\pm 1}] \amalg_R R[ (1{+}\,b)^{\pm 1}]$,}
which then embeds in $R\series{a,b}$.
We thus recover a celebrated result of \mbox{R.\,H.\,\,Fox},
via a proof simpler than those previously known.

We show that $\alpha$ is injective if $R$ is
\textit{$\Pi$-semihereditary,} that is,  every finitely generated,
torsionless, right $R$-module is projective.
 (This concept was first
studied by \mbox{M.\,F.\,\,Jones}, who showed that it is left-right
symmetric. It follows from a result of
\mbox{I.\,I.\,\,Sahaev} that
if  \,$\text{\normalfont{w.gl.dim\,}}R \le 1$ and $R$  embeds in a skew field,
then $R$~is $\Pi$-semihereditary.
Also, it follows from a result of
\mbox{V.\,C.\,\,Cateforis}
that if  $R$ is right semihereditary and  right self-injective,
then $R$ is $\Pi$-semihereditary.)

The arguments and results extend easily from two variables to any
set of variables.

The article concludes with some results contributed by
 G.\,M.\,\,Berg\-man that describe situations
where  $\alpha$ is not injective.
He shows that if $R$ is commutative and
\mbox{$\text{w.gl.dim\,} R \ge 2$,}
then there exist examples where the map
\mbox{$\alpha' \colon A \amalg_R B \to R\series{a}\amalg_R R\series{b}$}
is not injective, and hence neither is $\alpha$.  It follows from a result of  K.\,R.\,\,Goodearl  that
when $R$ is a commutative, countable, non-self\d1in\-jec\-tive, von
Neumann regular ring, the map
\mbox{$\alpha''\colon R\series{a}\amalg_R R\series{b}
 \to  R\series{a,b}$} is not injective.
\mbox{Bergman}  gives  procedures for constructing
other examples where $\alpha''$ is not injective.

\medskip
\medskip

{\footnotesize

}
\end{abstract}

\section{Summary of the results}\label{sec:res}

Some of the terms that we use in this section will be explained
in more detail in subsequent sections.

Throughout, we fix a ring $R$ (associative, with 1).
By an \textit{$R$-ring} $S$ we mean a ring $S$ together with
a specified ring homomorphism $R \to S$.
We write $\text{U}(R)$ for the multiplicative group of units of $R$.
We shall be considering the following.

\begin{hypotheses}\label{hyp:main}
Let $R\series{a,b}$ denote the power-series $R$-ring
in two non-com\-muting, $R$-centralizing variables, $a$~and~$b$.
Let $A$ be an $R$-subring of $ R\series{a}$  and $B$~be an $R$-subring
of $R\series{b}$.  Let $\alpha$ denote the natural map
\mbox{$A \amalg_R B \to R\series{a,b}$}.
\end{hypotheses}

\medskip

This article describes some situations where $\alpha$ is injective and some
 where  it is not.

\medskip

In Section~\ref{sec:not}, we fix some of the notation, and we
use a result of P.\,M.\,\,Cohn to decompose the map from
the domain of $\alpha$ to the image of $\alpha$
as a direct sum of $R$-bimodule maps.

\medskip

In Section~\ref{sec:flat}, in Proposition~\ref{prop:loc}, we shall see
that if $A$ is a right Ore localization of $R[a]$ and $B$ is a
right Ore localization of $R[b]$, then $\alpha$ is injective.

\medskip

\begin{example}\label{ex:Fox}
Let \mbox{$h\coloneq 1{+}\,\,a$}, \mbox{$H \coloneq \gen{h}\le\text{U}(R\series{a})$},
and let $RH$ denote the group ring of $H$ over $R$.
It is not difficult to see that we may view $R[a]$ as
\mbox{$R[h] \subseteq RH \subseteq R\series{a}$} and that $RH$ is
a central Ore localization of~\mbox{$R[h]$}.
The analogous statements hold for \mbox{$k \coloneq 1{+}\,\,b$} and
\mbox{$K \coloneq \gen{k} \le \text{U}(R\series{b})$}.
By Proposition~\ref{prop:loc} below, the map
\mbox{$\alpha \colon RH \amalg_R RK \to R \series{a,b}$} is injective.
Here, \mbox{$RH \amalg_R RK $} may be identified with the group ring  over
$R$ of the  group~$F$ freely generated by  two symbols   $1{+}\,\,a, 1{+}\,\,b$.
We thus recover
the celebrated result given in 1953 by R.\,H.\,\,Fox~\cite{Fox53}, that
 $RF$   embeds in $R\series{a,b}$.
Although Fox considers only the case $R = \integers$, his argument works for any ring.
The proof given here is simpler
than the previously known proofs, such as the original proof reviewed recently in \mbox{\cite[pp.\,366--371]{AD}},
and the proof by G.\,M.\,\,Bergman~\cite[pp.\,528--529]{Cohn85} of the case where $R$ is a field.

(One consequence of Fox's result is that  if $R \ne \{0\}$, then the
 group $F$ embeds in $\text{U}(R\series{a,b})$.
\mbox{W.\,Magnus~\cite{Magnus35}}  had already given a short proof of the
$R = \integers$ case of this consequence in 1935,  and deduced much useful information about free groups.
Some of this information eventually led  to the discovery, in 1946, of the important fact
that $F$ is an orderable group.
Much later, Bergman~\cite{Bergman90}  remarked that orderability follows directly
from Magnus' embedding, since
$\integers\series{a,b}$ can easily be given a ring ordering whose positive cone contains
$F$.)
\end{example}

In Section~\ref{sec:back}, we review some definitions and results from
ring theory.  We say that $R$ is \textit{$\Pi$-semihereditary} if
every finitely generated, right $R$-submodule of a direct product of copies of $R$
is projective.  (Completely reducible rings are
clearly $\Pi$-semihereditary, since all their modules are projective.
It is also clear that $\Pi$-semihereditary rings are right semihereditary and, hence,
satisfy \mbox{$\text{w.gl.dim\,} R \le 1$}.)
This concept was studied first by
M.\,F.\,\,Jones~\cite{Jones} who
showed that it is left-right symmetric.
Earlier,  K.\,R.\,\,Goodearl~\cite{Goodearl72} had observed that
 a result of
V.\,C.\,Cate\-foris~\cite{Cateforis} shows
that if $R$ is  right semihereditary and right self-injective,
then $R$ is   $\Pi$-semihereditary; see Theorem~\ref{thm:Cateforis} below.
It follows from a result of
I.\,I.\,\,Sahaev~\cite{Sahaev}
that if  \mbox{$\text{w.gl.dim\,} R \le 1$} and $R$ embeds in a skew
field, then $R$ is $\Pi$-semihereditary;
see Corollary~\ref{cor:JondrupMain}\eqref{it:J1}
 below.
We recall these and other examples of $\Pi$-semihereditary rings.

\medskip

 In Section~\ref{sec:PID}, we write $\alpha$ as a composite
\mbox{$A \amalg_R B \stackrel{\alpha'}{\to}
R\series{a}\amalg_R R\series{b} \stackrel{\alpha''}{\to} R\series{a,b}.$}
Then, in Corollary~\ref{cor:suff0},
we show that $\alpha'$, $\alpha''$, and $\alpha$ are injective if
$R$ is $\Pi$-semi\-hered\-i\-tary, or, more generally,
if every finitely generated right $R$-submodule
of $R\series{a}$ is projective.
The examples of the preceding section then apply.

\medskip

The results of Sections~\ref{sec:flat} and~\ref{sec:PID} extend from
two variables to any set of variables; the arguments are very similar
but the notation is more complicated.
We leave the details to the interested reader.

\medskip

Section~\ref{sec:ex} contains   results contributed by
Berg\-man that describe situations
where  $\alpha$ is not injective.
Examples of non-injectivity of $\alpha$ can arise from
non-injectivity either of $\alpha'$ or of $\alpha''$.
(Examples of the former sort clearly lead to non-injectivity
of~$\alpha$, while an example of the latter sort does
so on taking \mbox{$A=R\series{a}$}, \mbox{$B=R\series{b}$}.)
Examples are given of both sorts.
He obtains non-injectivity of $\alpha'$ for any
commutative ring $R$ with \mbox{$\text{w.gl.dim\,} R \ge 2$}, e.g.
  \mbox{$R=\rationals[x,y]$}.
A result of \mbox{Goodearl~\cite{Goodearl72}}
yields non-injectivity of $\alpha''$
whenever $R$ is a commutative, countable, non-self\d1in\-jec\-tive, von
Neumann regular ring,  e.g.
 \mbox{$R=\rationals[e_i : i \in \integers]/(e_i e_j - \delta_{i,j} e_i : i,j \in \integers).$}
Constructions are described that generalize  the latter example.

\section{The bimodule structure of ring coproducts}\label{sec:not}

\begin{notation}
We denote by $\naturals$ the set of finite cardinals,
$\{0,1,2,3,\ldots\}$.

For any $R$-bimodules $M_1$ and $M_2$, we will denote the $R$-bimodule
$M_1 \otimes _R M_2$ by $M_1 \otimes M_2$ if $R$ is understood.

For any $R$-bimodule $M$ and any $n \in \naturals$, we recursively
define the $R$-bimodule \textit{tensor power} $M^{\otimes n}$ by
the formulas $M^{\otimes 0}:= R$ and
$M^{\otimes (n{+}1)}:= M^{\otimes n} \otimes M$.
\end{notation}

\begin{notation}\label{not:2} Suppose that Hypotheses~\ref{hyp:main} hold.

Let $\{a,b\}^\ast$ denote the free monoid on $\{a,b\}$.
We shall usually write elements of $R\series{a,b}$ as
formal sums \hskip-10pt$\sum\limits_{w \in \{a,b\}^*} \hskip-7pt f(w){\cdot}w,$
where $f\colon w\mapsto f(w)$ is an arbitrary function from $\{a,b\}^\ast$
to $R$.
The set $\{w \in \{a,b\}^*  \colon f(w) \ne 0\}$ is called the
\textit{$\{a,b\}^*$-support} of
\hskip-10pt$\sum\limits_{w \in \{a,b\}^*}\hskip-7pt f(w){\cdot}w.$

We view $R \subseteq R[a] \subseteq R\series{a} \subseteq R\series{a,b}$
and \mbox{$R \subseteq R[b] \subseteq R\series{b} \subseteq R\series{a,b}$}.
Let  \mbox{$\mathfrak{a} \coloneq A \cap aR\series{a}$} and
$\mathfrak{b} \coloneq B \cap bR\series{b}$.
Then $\mathfrak{a}$ is a two\d1sided ideal of $A$, and,
as $R$-bimodules, $A = R \oplus \mathfrak{a}$.
Analogous statements hold for $B$.

The ring coproduct of $A$ and $B$ amalgamating the two copies of $R$
will be denoted $A \amalg_R B$, or by $A \amalg B$ if $R$ is understood.
Now $A \amalg B = (R\oplus \mathfrak{a}) \amalg (R\oplus \mathfrak{b})$.
As noted by P.\,M.\,\,Cohn \cite[pp.\,60--61]{Cohn64},
it is not difficult to show that
there is then an expansion as a direct sum of $R$-bimodules
$$ A \amalg B = R \oplus \mathfrak{a} \oplus \mathfrak{b}
\oplus (\mathfrak{a}{\otimes} \mathfrak{b})
\oplus (\mathfrak{b} {\otimes} \mathfrak{a})
\oplus (\mathfrak{a}{\otimes} \mathfrak{b} {\otimes} \mathfrak{a})
\oplus (\mathfrak{b}{\otimes} \mathfrak{a} {\otimes} \mathfrak{b})
\oplus \cdots.$$
We may write this in the form
\begin{equation}\label{eq:sum0}
A \amalg B = R \oplus \mathfrak{a} \oplus \mathfrak{b}
\oplus{\textstyle \bigoplus \limits_{n\ge 1}}
\Bigl((\mathfrak{a} {\otimes} \mathfrak{b})^{{\otimes} n}
\oplus (\mathfrak{b} {\otimes} \mathfrak{a})^{{\otimes} n}
\oplus ((\mathfrak{a} {\otimes} \mathfrak{b})^{{\otimes} n}
{\otimes} \mathfrak{a})
\oplus ((\mathfrak{b} {\otimes} \mathfrak{a})^{{\otimes} n}
{\otimes} \mathfrak{b})\Bigr).
\end{equation}

When $\alpha$ is applied to~\eqref{eq:sum0}, the direct sums are
replaced with sums, and the tensor-product multiplications within
$A \amalg B $ are replaced with the
multiplications of $R\series{a,b}$.
Thus, we may write
\begin{equation}\label{eq:sum1}
\alpha(A \amalg B) =
R + \mathfrak{a} + \mathfrak{b} +
{\textstyle\sum \limits_{n\ge 1}}\Bigl((\mathfrak{a}\mathfrak{b})^{ n}
+ (\mathfrak{b} \mathfrak{a})^{ n}
+((\mathfrak{a} \mathfrak{b})^{ n} \mathfrak{a})
+ ( (\mathfrak{b} \mathfrak{a})^{ n} \mathfrak{b})\Bigr).\vspace{-2mm}
\end{equation}
The $\{a,b\}^\ast$-support of $\mathfrak{a}$ lies in
$\{a^i : i \ge 1\}$, and the
$\{a,b\}^\ast$-support of $\mathfrak{b}$ lies in $\{b^j : j \ge 1\}$.
It follows that any two summands appearing in~\eqref{eq:sum1}
have disjoint $\{a,b\}^\ast$-supports.
Hence, we recover directness in the summation, and may write
\begin{equation*}
\alpha(A \amalg B) =
R \oplus \mathfrak{a} \oplus \mathfrak{b}
\oplus {\textstyle\bigoplus \limits_{n\ge 1}}
\Bigl((\mathfrak{a}\mathfrak{b})^{n}
\oplus (\mathfrak{b} \mathfrak{a})^{n}
\oplus((\mathfrak{a} \mathfrak{b})^{n} \mathfrak{a})
\oplus ((\mathfrak{b}\mathfrak{a})^{n} \mathfrak{b})\Bigr).\vspace{-2mm}
\end{equation*}

For any finite alternating string $w$   formed from the
letters $a$ and $b$, let us write $\alpha_w$ for the map
from the corresponding alternating tensor product of factors
$\mathfrak{a}$~and~$\mathfrak{b}$ in $A \amalg B$ to
the corresponding alternating product in $R\series{a,b}$.
Thus, as an $R$-bimodule map, $A \amalg B \to \alpha(A \amalg B) $
decomposes as the direct sum of the three identity maps
$R \xrightarrow{\alpha_\emptyset} R$,
$\mathfrak{a} \xrightarrow{\alpha_a} \mathfrak{a}$,
$\mathfrak{b} \xrightarrow{\alpha_b} \mathfrak{b}$, and the further $R$-bimodule maps, $n \ge 1$,
\begin{align*}
(\mathfrak{a} {\otimes} \mathfrak{b})^{\otimes n} &\xrightarrow{\alpha_{(ab)^n}} (\mathfrak{a}\mathfrak{b})^n,  &
(\mathfrak{b} {\otimes} \mathfrak{a})^{\otimes n} &\xrightarrow{\alpha_{(ba)^n}} (\mathfrak{b}\mathfrak{a})^n,  \\
(\mathfrak{a} {\otimes} \mathfrak{b})^{\otimes n} \otimes \mathfrak{a} &\xrightarrow{\alpha_{(ab)^n a}}
(\mathfrak{a}\mathfrak{b})^n\mathfrak{a}, &
(\mathfrak{b} {\otimes} \mathfrak{a})^{\otimes n} \otimes \mathfrak{b} &\xrightarrow{\alpha_{(ba)^n b}}
(\mathfrak{b}\mathfrak{a})^n\mathfrak{b}.
\end{align*}
Thus $\alpha$ is injective if and only if each of the four maps is  injective for all $n \ge 1$.
\end{notation}

\section{Ore localizations}\label{sec:flat}

We can now give a simple condition that ensures that $\alpha$ is injective.

\begin{proposition}\label{prop:loc} If
{\hskip1mm\normalfont Hypotheses~\ref{hyp:main}} hold and, moreover,
$A$ is a right Ore localization of $R[a]$ and $B$ is a right Ore
localization of $R[b]$, then $\alpha$ is injective.
\end{proposition}

\begin{proof} We use Notation~\ref{not:2}.
We shall show below that for any $n\ge 1$, if $\alpha_{(ab)^n}$
is injective, then $\alpha_{(ab)^n a}$ is injective.
The analogous arguments will clearly work for the other three cases
required.

Suppose then that $   (\mathfrak{a} \otimes
\mathfrak{b})^{\otimes n} \xrightarrow{\alpha_{(ab)^n}} (\mathfrak{a}\mathfrak{b})^n
\subseteq R\series{a,b}$ is injective.  By considering $\{a,b\}^*$-supports,
we see that, for each $d \ge 0$,
 $(\mathfrak a\otimes\mathfrak b)^{\otimes n} {\otimes}\, Ra^d \to
(\mathfrak a\mathfrak b)^n  a^d$  is then injective.   Again by considering $\{a,b\}^\ast$-supports,
we see that the map
$${\textstyle \bigoplus\limits_{d \ge 0}}
\Bigl( (\mathfrak{a} {\otimes} \mathfrak{b})^{{\otimes} n} \otimes
(Ra^d )\Bigr)
\to {\textstyle\bigoplus\limits_{d \ge 0}}
\Bigl((\mathfrak{a}\mathfrak{b})^na^d\Bigr)\subseteq R\series{a,b}$$
is injective; that is, the map
$$(\mathfrak{a} {\otimes} \mathfrak{b})^{{\otimes} n} \otimes (R[a])
\to (\mathfrak{a}\mathfrak{b})^n(R[a])\subseteq R\series{a,b}$$
is injective.

Since $A$ is a right Ore localization of $R[a]$, there exists
some subset $U$ of $ \text{U}(A)$ such that
the family $(R[a]u : u \in U)$ forms a directed system under
inclusion and the direct limit of the system is~$A$.
For each $u \in U$,  we have a commutative diagram
$$\begin{CD}
\textstyle(\mathfrak{a} {\otimes} \mathfrak{b})^{{\otimes} n} \otimes (R[a])
 @>>>\textstyle  (\mathfrak{a}\mathfrak{b})^n (R[a]) \\[-5pt]
@VVuV    @VVuV \\
\textstyle(\mathfrak{a} {\otimes} \mathfrak{b})^{{\otimes} n} \otimes  (R[a]u)
 @>>>\textstyle  (\mathfrak{a}\mathfrak{b})^n(R[a]u)
\end{CD}$$
in which the  maps on the top and the two sides are bijective.  Hence, the map on the bottom,
$$(\mathfrak{a} {\otimes} \mathfrak{b})^{{\otimes} n} \otimes (R[a]u)
\to (\mathfrak{a}\mathfrak{b})^n(R[a]u)\subseteq R\series{a,b},$$
is injective.
On taking the directed union over $u \in U$, we find that the map
$$(\mathfrak{a} \otimes \mathfrak{b})^{\otimes n}\otimes A
\to (\mathfrak{a}\mathfrak{b})^nA \subseteq R\series{a,b}$$
is injective.
By considering $\{a,b\}^\ast$-supports,
we see that the latter map is the direct sum of the two maps
$\alpha_{(ab)^n}$ and $\alpha_{(ab)^n a}$.
Hence, $\alpha_{(ab)^n a}$ is injective.

The other three inductive implications can be proved in the same way.
The injectivity of $\alpha$ then follows by induction.
\end{proof}

\begin{remarks} In Example~\ref{ex:Fox}, we noted that
one case of Proposition~\ref{prop:loc} is the result of
Fox~\cite{Fox53} that the natural map $RF \to R\series{a,b}$ is
injective, where $F$ is the free group on two symbols $1{+}\,a$, $1{+}\,b$.
Fox's proof  introduced $R$-linear endomorphisms of $RF$ that are now called
\textit{Fox derivatives}; see~\cite[Example 7, p.\,55]{Cohn85}.
To indicate one important connection between the above proof of
Proposition~\ref{prop:loc} and Fox's original argument, we
remark that right multiplication by $a$ gives a bijective map
$$(\mathfrak{a} {\otimes} \mathfrak{b})^{{\otimes} n} \otimes (Ra^{d} )
\to
(\mathfrak{a} {\otimes}\mathfrak{b})^{{\otimes} n}\otimes (Ra^{d+1} ),$$
and the inverse map has the same action as that of the left Fox
derivative on $RF$ with respect to $1{+}\,a$ relative to the free
generating set $\{1{+}\,a,1{+}\,b\}$.

The original motivation of this article
was to show that the proof of the embedding became more transparent
when more emphasis was placed on the ring-theoretic viewpoint.
\end{remarks}

\section{Review of $\textstyle\Pi$-semihereditary rings}\label{sec:back}

In this section, we review some of the history of what we call
$\Pi$-semihereditary rings.
This will provide us with some examples
to which the results of the next section apply.
We shall recall the proofs since they are all fairly short, scattered over the literature,
and  not always available in the formulation that will be required.
No new results are claimed, but we have tried to give proofs that minimize background
requirements.

\begin{definitions}\label{defs:early}
For each of the definitions given below, the dual left-right definition
 will also be understood.

 Let $M$ be a left $R$-module.

(1) Let $X$ and $Y$ be sets.

For  $m,\, n \in \naturals$, we let $\null^m\mkern-1mu Y^n$ denote the
set of \mbox{$m\times\,n$} matrices with coordinates in~$Y$. Absence of an exponent is to be read as an exponent~$1$.

Let $\prod\limits_{X} Y$  denote the set whose elements are the functions $X \to Y$, $x \mapsto y_x$,
represented as families $(y_x : x \in X)$.
We shall sometimes write $\prod\limits_{X} Y$ as  $Y^X$, and think of the
elements intuitively as `$1 \times X$ matrices'.  Similarly, we sometimes
write $\prod\limits_{X} Y$ as  $\null^X \mkern-2mu Y$, and think of the
elements intuitively as `$X \times 1$ matrices'.

In a natural way, $\prod\limits_X  R$ is an $R$-bimodule and
 $\prod\limits_X  M$ is a left $R$-module.
 By the \textit{multiplication map} \mbox{$(\prod\limits_X  R) \otimes
M \to (\prod\limits_X  R) M \subseteq  \prod\limits_X M$},
we shall mean the left $R$-module map determined by
\mbox{$(r_x:x \in X) \otimes m \mapsto (r_xm: x \in X).$}
We shall sometimes let \mbox{$(\null^X \mkern-2mu R) \otimes
M \to (\null^X \mkern-2mu R) M \subseteq  \null^X \mkern-2mu M$}
denote  the underlying additive map; here, $  \null^X \mkern-2mu R $ has a right $R$-module structure and
$\null^X \mkern-2mu M$ has an abelian group structure.  When $M$ is an $R$-bimodule, the right $R$-action is
to be maintained in all cases.

When $M =  \prod\limits_Y  R$, we  let $\mu_{X,Y}$ denote the foregoing multiplication map.
Here, we have an  $R$-bimodule map
\begin{align*}
&\hskip1.8cm\textstyle(\prod\limits_{X} R) \otimes (\prod\limits_{Y} R) \xrightarrow{\mu_{X,Y}}
\prod\limits_{X  \times Y } R,  \\
&(r_x : x \in X )
\otimes (s_y : y \in Y ) \mapsto (r_x s_y : (x,y) \in X  \times Y ).
\end{align*}
We shall sometimes denote the underlying additive map by
$\null^X\mkern-2mu R \otimes R^Y \xrightarrow{\mu_{X,Y}} \null^X\mkern-2mu R^Y$,
where we think of the elements  of $\null^X\mkern-2mu R^Y$ imprecisely
as `$X \times Y$ matrices'.

(2) We say that $M$ is \textit{torsionless}
if there exists an injective left $R$-module map from
$M$ to a direct product of copies of $R$.
Clearly, submodules of torsionless modules are torsionless.

(3)  One says that $M$ is \textit{coherent} if every finitely generated
submodule is finitely presented.
Clearly, submodules of coherent modules are coherent.

The ring $R$ is \textit{left coherent} if $R$ is coherent as left
$R$-module, that is, every finitely generated left ideal of $R$
is finitely presented as left $R$-module.

The ring $R$ is \textit{left $\Pi$-coherent} if, for every set~$Y$,
$R^Y$ is coherent as left $R$-module, or, equivalently,
every finitely generated, torsionless, left $R$-module is
finitely presented.
This property was studied first by Jones~\cite[p.\,103]{Jones}.
Clearly, left $\Pi$-coherent rings are left coherent.

(4)  We say that $M$ is \textit{semihereditary} if every finitely
generated submodule is projective.
Clearly, semihereditary modules are coherent, and submodules of semihereditary modules are semihereditary.

The ring $R$ is \textit{left semihereditary} if $R$ is semihereditary as left $R$-module, that is,
every finitely generated left ideal of $R$ is projective as left $R$-module.

We now come to the property of interest to us, which was studied first
by Jones~\cite{Jones},
who showed that it is left-right symmetric; see Theorem~\ref{thm:equiv} below.
We say that the ring $R$ is \textit{$\Pi$-semihereditary}
if, for every set $X$, $\null^X\mkern-2mu R$ is semihereditary as right $R$-module, or, equivalently,
every finitely generated, torsionless right $R$-module is projective.
Prior to Jones' work, Goodearl~\cite{Goodearl72} had mentioned that right self-injective, von
Neumann regular rings have this property; see Theorem~\ref{thm:Cateforis} below.

Clearly, $\Pi$-semihereditary rings are right semihereditary and right $\Pi$-coherent.

(5) We say that $M$ is \textit{flat} if, for all injective right
 $R$-module maps $K \to F$, the
induced map $K \otimes_RM \to F \otimes_RM$ is injective, or, equivalently,
for all finitely generated right ideals $I$ of $R$,
the multiplication map \mbox{$I \otimes_R M \to IM \subseteq M = R\otimes_R M$} is injective.
See~\cite[Theorem~3.53]{Rotman}.  Notice that it is therefore sufficent to consider
inclusion maps $K \to F$ where $F$ is a finitely generated, free, right $R$-module and
$K$ is a finitely generated $R$-submodule of $F$.

If every finitely generated submodule of $M$ is flat, then $M$ is flat. More generally, if
$M$ is the direct  limit of a directed system  of flat modules, then $M$ is flat.
 See~\cite[Theorem~3.47]{Rotman}.

If $M$ is flat, then every left $R$-summand of $M$ is flat. See~\cite[Theorem~3.45]{Rotman}.

We write $\text{w.gl.dim\,}R \le 1$ if every finitely generated left
ideal of $R$ is flat as left $R$-module, or,
equivalently, every submodule of every flat left $R$-module is flat,
or equivalently every submodule of every flat right $R$-module is flat; see~\cite[Theorem~9.24]{Rotman}.
Otherwise, we write $\text{w.gl.dim\,}R \ge 2$.

(6)  For any infinite cardinal $\kappa$, we say that $R$ is \textit{left $\kappa$-Noetherian} if every
left ideal of $R$ is generated by $\kappa$ (or fewer) elements.  For example, if $\vert R \vert \le \aleph_0$, then
$R$ is left $\aleph_0$-Noetherian.  Also, if  $\vert R \vert \ge \aleph_0$, then
$R$ is left $\vert R \vert$-Noetherian.
\end{definitions}

The next result is a variant of~\cite[Theorem~2.4]{Cohn59}   suitable for our purposes.

\begin{theorem}\label{thm:flatpre2}  {\normalfont (Cohn, 1959)}    For any
flat left $R$-module $F$ and
 left $R$-submodule $K$ of $F$,  the following are equivalent.\vspace{-2mm}
\begin{enumerate}[{\rm (a)}]
\setlength\itemsep{-2pt}
\item \label{it:Flat0pre2} $ F/K$  is left $R$-flat.
\item \label{it:Flat3pre2}  For all  $\ell$, $m  \in \naturals$  and
$y \in \null^\ell \mkern-2mu R^m$, \mbox{$(y  (\null^mF))\cap   (\null^\ell\mkern-2mu K) \subseteq
y(\null^m \mkern-2mu K)$} in $\null^\ell F$.
\end{enumerate}
\end{theorem}

\begin{proof} If $\ell$, $m  \in \naturals$  and
$y \in \null^\ell \mkern-2mu R^m$, then $\null^\ell \mkern-2mu R$ may be considered an
arbitrary  finitely generated, free, right $R$-module  and $y(\null^m \mkern-2mu R)$
 an arbitrary finitely generated, right $R$-submodule  of $\null^\ell \mkern-2mu R$.
Thus, the left $R$-module  $M \coloneq  F/K$ is  flat if and only if
the natural map\vspace{-2mm}   \begin{equation}\label{eq:map}
 (y(\null^m \mkern-2mu R)) \otimes_R M
 \to (\null^\ell \mkern-2mu R)\otimes_R M\vspace{-2mm}  \end{equation}
is injective  for all  \mbox{$\ell$, $m$, $y$}.

 Suppose that $\ell$, $m  $,  and
$y$ are given.

We have an exact sequence\vspace{-2mm}  \begin{equation}\label{eq:map2} (y(\null^m \mkern-2mu R))\otimes_R K
\to (y(\null^m \mkern-2mu R))\otimes_R F \to(y(\null^m \mkern-2mu R))\otimes_R M \to 0,  \vspace{-2mm} \end{equation}
and we shall re-interpret each term. Since $F$ is left $R$-flat, \mbox{$(y(\null^m \mkern-2mu R))  \otimes_R F$} embeds in
\mbox{$(\null^\ell \mkern-2mu R) \otimes_RF =  \null^\ell \mkern-2mu F$}, and
the image is $ y(\null^m \mkern-2mu F)$.  Hence, we may make the identification
 \mbox{$(y(\null^m \mkern-2mu R))  \otimes_R F = y(\null^m \mkern-2mu F)$}.
The image of \mbox{$(y(\null^m \mkern-2mu R))  \otimes_RK$} in
 \mbox{$(y(\null^m \mkern-2mu R))  \otimes_R F = y(\null^m \mkern-2mu F)$}  is  then
 \mbox{$y(\null^m \mkern-2mu K)$}.
  Using~\eqref{eq:map2}, we now find that we may  make the identification
  \mbox{$(y(\null^m \mkern-2mu R)) \otimes_R M
=  (y(\null^m \mkern-2mu F))/(y(\null^m \mkern-2mu   K))$}.

In the special case where $y$ is an identity matrix,
we  have the (obvious) identification
\mbox{$(\null^\ell \mkern-2mu R) \otimes_R M =
(\null^\ell \mkern-2mu F)/(\null^\ell \mkern-2mu   K)$}.
We now find that we may  identify~\eqref{eq:map}
  with the natural map
\mbox{$(y(\null^m \mkern-2mu F))/(y(\null^m \mkern-2mu   K)) \to
(\null^\ell \mkern-2mu F)/(\null^\ell \mkern-2mu   K)$}.
Then \eqref{it:Flat0pre2}$\mkern2mu\Leftrightarrow$\eqref{it:Flat3pre2} follows.
\end{proof}

The next two results are taken from~\cite[Proposition~2.2 and Corollary]{Chase},
attributed to  O.\,E.\,\,Villamayor.

\newpage

\begin{theorem}\label{thm:flat}  {\normalfont (Villamayor, 1960)}  For any  $n \in \naturals$ and
 left $R$-submodule $K$ of $R^n$,  the following are equivalent.\vspace{-2mm}
\begin{enumerate}[{\rm (a)}]
\setlength\itemsep{-2pt}
\item \label{it:Flat1} $R^n\mkern-3mu/K$ is left $R$-flat.
\item \label{it:Flat3}   For all $\ell \in \naturals$
and  $x  \in \null^\ell \mkern-2mu K \subseteq  \null^\ell \mkern-2mu R^n$,
 \mbox{$x  \in x (\null^n K)$}.
\end{enumerate}
\end{theorem}

\begin{proof}   The case  $F = R^n$  of
 Theorem~\ref{thm:flatpre2}   allows us to rewrite
\eqref{it:Flat1}  as

\noindent (\ref{it:Flat1}$'$)\hskip5pt{\em For all  $\ell$, $m  \in \naturals$  and
$y \in \null^\ell \mkern-2mu R^m$,
\mbox{$(y  (\null^mR^n))\cap   (\null^\ell\mkern-2mu K) \subseteq
y(\null^m \mkern-2mu K)$} in $\null^\ell R^n$.}

(\ref{it:Flat1}$'$)$\Rightarrow$\eqref{it:Flat3}.  Suppose $x \in  \null^\ell \mkern-2mu K$. Taking $m = n$ and $y = x$ in~(\ref{it:Flat1}$'$)  we obtain
\mbox{$(x  (\null^nR^n))\cap   (\null^\ell\mkern-2mu K) \subseteq
x(\null^n \mkern-2mu K)$} in $\null^\ell R^n$.  Then   $x   =    x  \mathbf{I}_n  \in   x   (\null^n\mkern-2mu R^n) \cap
(\null^\ell\mkern-2mu K)   \subseteq  x (\null^n \mkern-2mu K)$.

\mbox{\eqref{it:Flat3}$\Rightarrow$(\ref{it:Flat1}$'$)}. Suppose $y \in \null^\ell \mkern-2mu R^m$. Consider any  \mbox{$x \in
(y  (\null^m\mkern-2mu R^n))\cap   (\null^\ell\mkern-2mu K)$}.  We have \mbox{$x \in \null^\ell\mkern-2mu K$},
and then, by~\eqref{it:Flat3},
\mbox{$x \in (x)(\null^n \mkern-2mu K)  \subseteq
(y(\null^m\mkern-2mu R^n))(\null^n \mkern-2mu K)\subseteq y(\null^m\mkern-2mu K)$}.
\end{proof}

\begin{theorem}\label{thm:Vil} {\normalfont (Villamayor, 1960)}
For any left $R$-mod\-ule~$M$, the following are equivalent.\vspace{-2mm}
\begin{enumerate}[{\rm (a)}]
\setlength\itemsep{-2pt}
\item \label{it:Vila} $M$ is finitely generated and projective.
\item \label{it:Vilb} $M$ is finitely presented and flat.
\end{enumerate}
\end{theorem}

\begin{proof} \eqref{it:Vila}$\Rightarrow$\eqref{it:Vilb} is clear.

\eqref{it:Vilb}$\Rightarrow$\eqref{it:Vila}.  Since $M$ is a finitely presented left $R$-module,
there exist  \mbox{$n \in \naturals$} and  a finitely generated
left $R$-submodule $K$ of $R^n$ such that $M = R^n\mkern-2mu/K$.
Since  $K$ is finitely generated,
 there exist  \mbox{$\ell \in \naturals$}
and $x \in \null^\ell\mkern-2mu K \subseteq  \null^\ell\mkern-3mu R^n$ such that
 \mbox{$K =  (R^{\ell})  x$}.
By Theorem~\ref{thm:flat} \eqref{it:Flat1}$\Rightarrow$\eqref{it:Flat3},
since $M$ is $R$-flat, \mbox{$x  \in  x(\null^n \mkern-2mu K)$}, that is,
  there exists some $e \in  \null^n \mkern-2mu K$ such that \mbox{$x  = x{\cdot} e$}.
Now the left $R$-module map $R^n \to K$, $y \mapsto y{\cdot}e$, acts as the identity on
$K$ because it does so on the rows of $x$, which form a generating set of~$K$.
 This proves that $R^n\mkern-2mu/K$ is projective.
\end{proof}

The following result is a variant of~\cite[Satz~2]{Lenzing} which
is the same as~\cite[Proposition~2.1]{Cox},   attributed to  M.\,Auslander.

\begin{theorem}\label{thm:Lenz} {\normalfont (Lenzing, 1969 {\footnotesize \&}\,Auslander, 1970)}
For any finitely generated left $R$-module $M$, the following are equivalent.
\vspace{-2mm}
\begin{enumerate}[{\rm (a)}]
\setlength\itemsep{-1.5pt}
\item \label{it:Lenza} $M$ is finitely presented.
\item \label{it:Lenzb} For every set~$X$, the multiplication map
\mbox{$(\null^X \mkern-2mu R) \otimes_R M \to   \null^X \mkern-2mu M$}
is injective.
\item \label{it:Lenzc}  For some infinite set $X$ such that
  $R$ is left $\vert X  \vert$-Noether\-ian, the multiplication map \mbox{$(\null^X \mkern-2mu R) \otimes_R
M \to   \null^X \mkern-2mu M$}
is injective.
\end{enumerate}
\end{theorem}

\begin{proof} \eqref{it:Lenza}$\Rightarrow$ \eqref{it:Lenzb}.  By~\eqref{it:Lenza}, there exist
$m$, $n \in \naturals$ and an exact sequence of left $R$-modules  \mbox{$R^m \to   R^n \to M \to 0$}.
 We then have a  commutative diagram
with exact rows:
$$\begin{CD}
\textstyle(\null^X \mkern-2mu R) \otimes_R (R^m)@>>>(\null^X \mkern-2mu R) \otimes_R (R^n)@
>>>\textstyle(\null^X \mkern-2mu R) \otimes_R M @>>>0\\[-5pt]
\hskip-.7cm\mu_{X,m}@|  \hskip-.7cm\mu_{X,n}@|  \hskip.2cm@VVV \\
\textstyle  \null^X \mkern-2mu R^m@>>>\textstyle  \null^X \mkern-2mu  R^n @
>>>  \null^X \mkern-2mu M@>>>0.\\[-3pt]
\end{CD}$$
Hence, $(\null^X \mkern-2mu R) \otimes_R M \to  \null^X \mkern-2mu M$ is injective (and surjective).

\eqref{it:Lenzb}$\Rightarrow$ \eqref{it:Lenzc} is clear.

\eqref{it:Lenzc}$\Rightarrow$ \eqref{it:Lenza}.  By hypothesis, there exist
 $n \in \naturals$ and an exact sequence of left $R$-modules $0 \to K \to R^n \to M \to 0$,
and it suffices to show that $K$ is finitely generated.
We  then have a  commutative diagram
with exact rows:
$$\begin{CD}
@.\textstyle(\null^X \mkern-2mu R) \otimes_R K @>>>\textstyle(\null^X \mkern-2mu R) \otimes_R (R^n)@
>>>\textstyle(\null^X \mkern-2mu R) \otimes_R M@>>>0 \\
@.\hskip.2cm@VVV  \hskip-.7cm\mu_{X,n}@|  \hskip.2cm@| \\
0@>>>\textstyle  \null^X \mkern-2mu K@>>>\textstyle  \null^X \mkern-2mu R^n@
>>>  \null^X \mkern-2mu M@>>>0,
\end{CD}$$
where $(\null^X \mkern-2mu R) \otimes_R M \to \null^X \mkern-2mu M$ is injective by~\eqref{it:Lenzc},
 and is here easily seen to be surjective.
Hence, $(\null^X \mkern-2mu R) \otimes_R K \to  \null^X \mkern-2mu K$ is surjective, that is,
$(\null^X \mkern-2mu R) K =  \null^X \mkern-2mu K$.
Since $R$ is left $\vert X \vert$-Noetherian by~\eqref{it:Lenzc},
$K$~is $\vert X \vert$-generated; hence, there exists some generating family
 \mbox{$(k_x : x \in X)$} which may be viewed as an
 element  of $ \textstyle \null^X\mkern-2mu K $,
and it then has the form \mbox{$\sum_{i=1}^m  (r_{x,i} : x \in X)   k'_i$}, since
$ \textstyle \null^X\mkern-2mu K  =  (\null^X \mkern-2mu R) K $.
Then, for each $x \in X$,  $k_x = \sum_{i=1}^m r_{x,i} k'_i$.
Hence, \mbox{$K = \sum_{x \in X} Rk_x \subseteq \sum_{i=1}^m Rk'_i \subseteq K$}.
Thus, $K$ is finitely generated.
\end{proof}

The next two results are taken from~\cite[Theorems~2.1 and 4.1]{Chase}.

\begin{theorem}\label{thm:Chase} {\normalfont (Chase, 1960)}
The following are equivalent.
\vspace{-2mm}
\begin{enumerate}[{\rm (a)}]
\setlength\itemsep{-2pt}
\item \label{it:Chasea} $R$ is left coherent, i.e. every finitely generated left ideal is finitely presented.
\item\label{it:Chasenew} For every  finitely generated left ideal $I$ of $R$  and every set $X$,
the multiplication map
\mbox{$(\null^X \mkern-2mu R) \otimes_R I
\to  (\null^X \mkern-2mu R)I \subseteq \null^X \mkern-2mu I \subseteq \null^X \mkern-2mu R = (\null^X \mkern-2mu R)\otimes_RR$} is injective.
\item \label{it:Chaseb} For every set $X$, $\null^X \mkern-2mu R$ is right $R$-flat.
\end{enumerate}\vspace{-2mm}
\end{theorem}

\begin{proof}[Proof\,\, {\normalfont(after Lenzing~\cite[Section~4]{Lenzing})}.]

\eqref{it:Chasea} \mbox{$
\overset{\text{Theorem~\ref{thm:Lenz}}}
{\Longleftarrow\mkern-3mu=\mkern-3mu=\mkern-3mu=\mkern-3mu=\mkern-3mu\Longrightarrow}$}  \eqref{it:Chasenew}
\mbox{$
\overset{\text{clear}}
{\Leftarrow\mkern-3mu=\mkern-6mu=\mkern-3mu\Rightarrow}$} \eqref{it:Chaseb}.
\end{proof}

\begin{theorem}\label{thm:BC} {\normalfont (Chase, 1960)} The following are equivalent.
\vspace{-2mm}
\begin{enumerate}[{\rm (a)}]
\setlength\itemsep{-2pt}
\item \label{it:sh} $R$ is left semihereditary, i.e. every finitely generated left ideal is projective.
\item\label{it:wd} $\text{\normalfont w.gl.dim\,}R \le 1$
and $R$ is left coherent.
\item\label{it:wd2} $\text{\normalfont w.gl.dim\,}R \le 1$ and, for every set $X$, $\null^X \mkern-2mu R$ is right $R$-flat.
\item \label{it:left} Every torsionless right $R$-module is flat.
\item \label{it:left2} Every finitely generated, torsionless, right $R$-module is flat.
\end{enumerate}
\end{theorem}

\begin{proof} \eqref{it:sh} \mbox{$
\overset{\text{Theorem~\ref{thm:Vil}}}
{\Longleftarrow\mkern-3mu=\mkern-3mu=\mkern-3mu=\mkern-3mu=\mkern-3mu\Longrightarrow}$}  \eqref{it:wd}
\mbox{$
\overset{\text{Theorem~\ref{thm:Chase}}}
{\Longleftarrow\mkern-3mu=\mkern-3mu=\mkern-3mu=\mkern-3mu=\mkern-3mu\Longrightarrow}$} \eqref{it:wd2}
\mbox{$
\overset{\text{clear}}
{\Leftarrow\mkern-3mu=\mkern-6mu=\mkern-3mu\Rightarrow}$} \eqref{it:left}
\mbox{$
\overset{\text{clear}}
{\Leftarrow\mkern-3mu=\mkern-6mu=\mkern-3mu\Rightarrow}$} \eqref{it:left2}.
\end{proof}

Goodearl~\cite[Theorem~1]{Goodearl72} generalized Theorem~\ref{thm:Lenz}
 to the non-finitely generated case
by showing that, for any left $R$-module $M$,
all of the multiplication maps \mbox{$(\null^X \mkern-2mu R) \otimes_R M
\to  \null^X \mkern-2mu M$} are injective
if and only if, for each finitely generated submodule $N$ of $M$, the inclusion map $N \to M$ factors
through a finitely presented module.
We record the special case we shall be using and its special proof.
Recall the meaning of $\mu_{X,Y}$ from Definitions~\ref{defs:early}(1).

\begin{theorem}\label{thm:Good} {\normalfont (Goodearl, 1972) }
For any non-empty set $Y$,    the following are equivalent.
\vspace{-2mm}
\begin{enumerate}[{\rm (a)}]
\setlength\itemsep{-2pt}
\item\label{it:Good11} $R^Y$ is coherent as left $R$-module.\vspace{-1mm}
\item\label{it:Good12}  $R$ is a left coherent ring, and, for every set $X$, $\mu_{X,Y}$
  is injective.
\item\label{it:Good13}  $R$ is a left coherent ring,  and, for some infinite set $X$
 such that $R$ is left $\vert X\vert$-Noether\-ian, $\mu_{X,Y}$  is injective.
\end{enumerate}
\end{theorem}

\begin{proof} Let $M \coloneq   R^Y$, and let $\mathcal{M}$ denote the directed family of
finitely generated left $R$-submodules of~$M$.

\eqref{it:Good11}$\Rightarrow$\eqref{it:Good12}.  Since $Y$ is non-empty, $R$ embeds in $M$, and, hence,
 $R$ is coherent as left $R$-module,
that is, $R$ is a left coherent ring.
 Let $N\in\mathcal{M}$. By~\eqref{it:Good11},
$N$ is finitely presented.  By Theorem~\ref{thm:Lenz} \mbox{\eqref{it:Lenza}$\Rightarrow$\eqref{it:Lenzb}},
 the multiplication map \mbox{$( \null^X\mkern-2mu R ) \otimes  N
\to   \null^X\mkern-2mu N \subseteq  \null^X\mkern-2mu M$}
is injective.
On taking the direct limit of the system of   injective maps
\mbox{$( \null^X\mkern-2mu R ) \otimes N \to  \null^X\mkern-2mu M$} over all $N \in \mathcal{M}$,
we see that
\mbox{$( \null^X\mkern-2mu R ) \otimes M \xrightarrow{\mu_{X,Y}}  \null^X\mkern-2mu  M$} is injective, as desired.

\eqref{it:Good12}$\Rightarrow$\eqref{it:Good13} is clear.

\eqref{it:Good13}$\Rightarrow$\eqref{it:Good11}
(after Herbera-Trlifaj~\cite[Corollary 2.11]{Herbera}\,).
Let $N \in  \mathcal{M}$ and consider the natural commutative diagram\vspace{-1mm}
$$\begin{CD}
\textstyle( \null^X\mkern-2mu R) \otimes N  @>>>\textstyle  \null^X\mkern-2mu N\\[-5pt]
@VVV     @VVV \\
\textstyle( \null^X\mkern-2mu R) \otimes M @>\mu_{X,Y}>>\textstyle    \null^X\mkern-2mu M.   \\[-3pt]
\end{CD}$$
By Theorem~\ref{thm:Chase}~\mbox{\eqref{it:Chasea}$\Rightarrow$\eqref{it:Chaseb}},
$ \null^X\mkern-2mu R$ is right
$R$-flat, and, hence, the left map is injective.
By~\eqref{it:Good13},  the bottom map is injective.  Hence, the top map is injective.
By
Theorem~\ref{thm:Lenz} \mbox{\eqref{it:Lenzc}$\Rightarrow$\eqref{it:Lenza}},
$N$ is finitely presented.  Thus,  $M$ is coherent as left $R$-module.
\end{proof}

The following result is a variant of~\cite[(4.5)]{Bass} that is suitable for our purposes.

\begin{theorem}  {\normalfont (Bass, 1960) } \label{thm:Y} Let
 $Y$ be an infinite set such that $R$ is right $\vert Y  \vert$-Noether\-ian.
If $M$ is a finitely generated, torsionless, left
 $R$-module, then $M$ embeds in $R^Y$.
Hence, $R$ is a left $\Pi$-coherent ring if and only if $R^Y$ is a coherent left $R$-module.
\end{theorem}

\begin{proof}   Since $M$ is a finitely generated left $R$-module, on dualizing  we find
that the right $R$-module
\mbox{$M^* \coloneq \text{Hom}_R(M, \null_R R)$} embeds in a finitely generated, free, right $R$-module.
Since $R$ is right $\vert Y  \vert$-Noetherian, $M^*$ is $\vert Y  \vert$-generated.
On dualizing  we find that the left $R$-module
\mbox{$M^{**} \coloneq \text{Hom}_R(M^*, R_R)$} embeds in $  R^Y$.
Since $M$ is torsionless, the double-dual map $M \to M^{**}$ is injective, and
$M$ embeds in $  R^Y$.
\end{proof}

 \begin{corollary}\label{cor:Herbera} The following are equivalent.
\vspace{-2mm}
\begin{enumerate}[{\rm (a)}]
\setlength\itemsep{-2pt}
\item\label{it:HT1} $R$ is a left $\Pi$-coherent ring,  i.e. every finitely generated,
torsionless, left module is  finitely presented.
\item\label{it:HT2} $R$ is a left coherent ring  and, for all sets $X$ and $Y\mkern-3mu,$  $\mu_{X,Y}$
  is injective.
\item\label{it:HT3}  $R$ is a left coherent ring   and, for some infinite sets $X$ and $Y$
such that $R$ is left $\vert X\vert$-Noe\-ther\-ian
and right $\vert Y\vert$-Noe\-ther\-ian, $\mu_{X,Y}$    is injective.
\end{enumerate}
\end{corollary}

\begin{proof}   This follows from Theorems~\ref{thm:Good} and~\ref{thm:Y}.
\end{proof}

We now recall the aspect of~\cite[Theorem 2.11]{Jones} that is of interest to us.

\begin{theorem}\label{thm:equiv} {\normalfont (Jones, 1982)} The following are equivalent.
\vspace{-2mm}
\begin{enumerate}[{\rm (a)}]
\setlength\itemsep{-2pt}
\item \label{it:newz} $R$ is left and right semihereditary, and, for all sets $X$, $Y$, $\mu_{X,Y}$ is injective.
\item\label{it:newy}  $R$ is left and right coherent,  $\text{\normalfont w.gl.dim\,}R \le 1$,  and,
for all sets $X$, $Y$, $\mu_{X,Y}$ is injective.
\item \label{it:b} $R$ is left coherent, $\text{\normalfont w.gl.dim\,}R \le 1$, and $R$ is right $\Pi$-coherent.
\item\label{it:newx}  Every finitely generated, torsionless,
 right $R$-module is flat and finitely presented.
\item \label{it:newc} $R$ is $\Pi$-semihereditary, i.e. every finitely generated, torsionless,
 right $R$-module is projective.
\end{enumerate}\vspace{-2mm}
These conditions are also equivalent to their left-right duals.
\end{theorem}

\begin{proof}  \eqref{it:newz} \mbox{$
\overset{\text{Theorem~\ref{thm:BC}}}
{\Longleftarrow\mkern-3mu=\mkern-3mu=\mkern-3mu=\mkern-3mu=\mkern-3mu\Longrightarrow}$}  \eqref{it:newy}
\mbox{$
\overset{\text{Corollary~\ref{cor:Herbera}}}
{\Longleftarrow\mkern-3mu=\mkern-3mu=\mkern-3mu=\mkern-3mu=\mkern-3mu=\mkern-3mu\Longrightarrow}$} \eqref{it:b}
\mbox{$
\overset{\text{Theorem~\ref{thm:BC}}}
{\Longleftarrow\mkern-3mu=\mkern-3mu=\mkern-3mu=\mkern-3mu=\mkern-3mu\Longrightarrow}$} \eqref{it:newx}
\mbox{$
\overset{\text{Theorem~\ref{thm:Vil}}}
{\Longleftarrow\mkern-3mu=\mkern-3mu=\mkern-3mu=\mkern-3mu=\mkern-3mu\Longrightarrow}$} \eqref{it:newc},
and
\eqref{it:newz}   is left-right symmetric.  \end{proof}

In the remainder of this section, we recall some types of rings that are $\Pi$-semihereditary.
We first present a  result from~\cite{Sahaev}.

  \begin{theorem} \label{thm:Sah}  {\normalfont (Sahaev, 1965)}
For each $n \in \naturals$  and each $y \in \null^n\mkern-3mu R^n$,
let $\overset{\null_\bullet}{y}$ denote \mbox{$ \mathbf{I}_n   {-}\, y$}.  The following are equivalent.
\vspace{-2mm}
\begin{enumerate}[{\rm (a)}]
\setlength\itemsep{-2pt}
\item\label{it:Sah1} There exists  a finitely generated,
non-finitely presented, flat, left $R$-module~$M$.
\item\label{it:Sah2} There exist   $n \in \naturals$ and a
function $\naturals \to \null^n\mkern-3mu R^n$, $k \mapsto y_k$,
such that,  for each $k \in \naturals$, \mbox{$y_k \overset{\null_\bullet}{y}_{k+1} \mkern-2mu =
\mathbf{0}_n  \ne y_{k+1}\overset{\null_\bullet}{y}_k$;}
equivalently, $y_ky_{k+1} = y_k$
and $y_{k+1}y_k \ne y_{k+1}$;
equivalently, \mbox{$\overset{\null_\bullet}{y}_k   \overset{\null_\bullet}{y}_{k+1}   =  \overset{\null_\bullet}{y}_{k+1} $},
and \mbox{$  \overset{\null_\bullet}{y}_{k+1}  \overset{\null_\bullet}{y}_k   \ne \overset{\null_\bullet}{y}_{k} $}.
\end{enumerate}
\end{theorem}

\begin{proof}
\eqref{it:Sah1}$\Rightarrow$ \eqref{it:Sah2}.
Since $M$ is finitely generated and not finitely presented, there exist $n \in \naturals$
and a non-finitely generated left $R$-submodule $K$ of $ R^n$ such that \mbox{$M = R^n\mkern-3mu/K$}.
We view  $R^n$  as an $(R, \null^n\mkern-3mu R^n)$-bimodule.

We shall recursively construct a   sequence  \mbox{$(y_k   : k \in \naturals)$} in
$ \null^n\mkern-2mu K \subseteq  \null^n\mkern-3mu R^n$.
Set
\mbox{$y_0 \coloneq \mathbf{0}_n \in   \null^n\mkern-2mu K$}.
We may now suppose that $k \in \naturals$ and
$y_k \in   \null^n\mkern-2mu K$.
Since $K$ is not finitely generated, it is easy to see that there
exists some    $x_{k} \in \null^{n+1}\mkern-2mu K\subseteq  \null^{n+1}\mkern-3mu R^n$  such that
$(R^{n})y_k   \subset (R^{n+1})x_{k} \subseteq K$.
By Theorem~\ref{thm:flat} \eqref{it:Flat1}$\Rightarrow$\eqref{it:Flat3},
since $M$ is flat, $x_k \in x_k (\null^n\mkern-2mu K)$, that is,
 there exists some  \mbox{$y_{k+1} \in   \null^n\mkern-2mu K$}
such that $x_{k} = x_{k} y_{k+1}$.  Hence, $x_{k}\overset{\null_\bullet}{y}_{k+1}  = \mathbf{0}_n$.
Now,  \mbox{$(R^{n})y_k \overset{\null_\bullet}{y}_{k+1}
 \subseteq    (R^{n+1})x_{k}\overset{\null_\bullet}{y}_{k+1}  = \{0\}$}, and we see that
 $y_k \overset{\null_\bullet}{y}_{k+1}   = \mathbf{0}_n$.

Also, $(R^{n})y_{k+1} y_k  \subseteq (R^{n})y_k   \subset (R^{n+1})x_{k} =
(R^{n+1})x_{k} y_{k+1}  \subseteq (R^{n})  y_{k+1},$
and we see that $y_{k+1} y_k   \ne y_{k+1}$.  Thus, $y_{k+1} \overset{\null_\bullet}{y}_k     \ne \mathbf{0}_n$.

This completes the recursive construction of the
sequence, and \eqref{it:Sah2} holds.

\eqref{it:Sah2}$\Rightarrow$ \eqref{it:Sah1}.
 Let $F$ denote the
 $(R, \null^n\mkern-3mu R^n)$-bimodule  $R^n$.
Then
$(Fy_k: k \in \naturals)$ is a family of left $R$-submodules of $F$.
For each $k \in \naturals$,  $Fy_k = Fy_ky_{k+1} \subseteq Fy_{k+1}$;
also $Fy_{k+2}  \overset{\null_\bullet}{y}_{k+1}   \ne \{0\} = Fy_k  \overset{\null_\bullet}{y}_{k+1} $,
and, hence,  $Fy_{k+2} \ne Fy_{k}$.
Thus,     $(Fy_{k}: k \in \naturals\,)$ is a non-stationary,
ascending chain of left $R$-submodules of $F$.

  Set
  $K \coloneq   \bigcup_{k =0}^\infty (Fy_{k}) $  and
$M \coloneq F/K$.  Then $M$ is a finitely generated, non-finitely presented,  left $R$-module.
 To prove \eqref{it:Sah1}, it suffices
to show that $M$ is left $R$-flat, and here we may apply Theorem~\ref{thm:flat}
 \eqref{it:Flat3}$\Rightarrow$ \eqref{it:Flat1}.
Let $\ell  \in \naturals$ and $x \in \null^\ell\mkern-2mu K$;
it remains to show that
 \mbox{$x  \in x (\null^n\mkern-2mu K)$}.
As  \mbox{$x \in \null^\ell \mkern-2mu K
 =  \bigcup_{k =0}^\infty (\null^\ell\mkern-2mu (Fy_{k}))$}, there exists some $k \in \naturals$ such that
$x   \in \null^\ell \mkern-1mu (Fy_{k})$. Then $x  \overset{\null_\bullet}{y}_{k+1}   = 0$.
Since each row of $y_{k+1}$ lies in~$K$,  $y_{k+1} \in \null^n\mkern-2mu K $.
Now,   $x  = x y_{k+1} \in  x   ( \null^n\mkern-2mu K)$, as desired.
\end{proof}

\begin{definitions}
(1) We shall say that $R$ is a \textit{left Sahaev ring} if
every finitely generated, flat, left $R$-module is finitely presented,
or, equivalently, projective, by Theorem~\ref{thm:Vil}.
The   term    `\hskip.1mm left S-ring'  is sometimes used in the literature.

Every subring of a left  Sahaev ring is again a left  Sahaev ring, since
these are the rings for which the condition in Theorem~\ref{thm:Sah}\eqref{it:Sah2} fails.
In particular, rings that embed in skew fields are left and right Sahaev rings.

It is clear from the foregoing definition that left Noetherian rings are left Sahaev rings.
Both Sahaev~\cite{Sahaev} and
S.\,\,J\o n\-drup~\cite[Lemma~1.4]{Jondrup3} discuss
connections between the condition in Theorem~\ref{thm:Sah}\eqref{it:Sah2}
and various chain conditions.

(2) A \textit{commutative Pr\"ufer domain} is a semihereditary, commutative domain.
 A \textit{commutative B\'ezout domain} is a commutative domain in which
every finitely generated ideal is principal.

(3) One says that $R$ is a \textit{semifir} if every finitely generated left ideal of $R$ is free of unique rank,
as left $R$-module, and similarly on the right.  It is sufficient to assume
the condition on the left; see~\cite[Theorem~I.1.1]{Cohn85}.
\end{definitions}

We now present~\cite[Corollary 2]{Jondrup3}.

 \begin{theorem} \label{thm:jondrup}{\normalfont(J\o ndrup, 1971)}  The following hold.
\vspace{-2mm}
\begin{enumerate}[{\rm (i)}]
\setlength\itemsep{-2pt}
\item \label{it:jon1} If $R$ is right semihereditary and  not left Sahaev, then $R$ is not right Sahaev.
\item \label{it:jon2}   If  \,$\text{\normalfont{w.gl.dim\,}}R \le 1$   and $R$~is   right Sahaev,
 then $R$ is  left Sahaev  and   $\Pi$-semi\-hereditary.
\end{enumerate}
\end{theorem}

\begin{proof} \eqref{it:jon1} We apply Theorem~\ref{thm:Sah} \eqref{it:Sah1}$\Rightarrow$\eqref{it:Sah2}, and
we change $y_0$ to $\mathbf{0}_n$.
 We shall speak of the
kernel  and image  of each element  of $\null^n\mkern-2mu R^n$ viewed as an $R$-endomorphism
of $\null^n\mkern-2mu R$ acting on the left by
matrix multiplication.

Let $k \in \naturals$.
Since \mbox{$y_{k+2} \overset{\null_\bullet}{y}_{k+1}   \ne \mathbf{0}_n = y_{k} \overset{\null_\bullet}{y}_{k+1} $},
  \mbox{$\operatorname{Ker}(y_{k+2}) \not\supseteq
\operatorname{Im}( \overset{\null_\bullet}{y}_{k+1} )
  \subseteq \operatorname{Ker}(y_{k})$}.  Hence,
$\operatorname{Ker}(y_{k+2}) \ne \operatorname{Ker}(y_k)$.
Clearly $\overset{\null_\bullet}{y}_{k+1} $ acts as the identity on $ \operatorname{Ker}(y_{k+1})$, and we see that
\mbox{$ \operatorname{Ker}(y_{k+1})  \subseteq
 \operatorname{Im}( \overset{\null_\bullet}{y}_{k+1} )
 \subseteq \operatorname{Ker}(y_k)$}.
Hence,  $\operatorname{Ker}(y_{2k+2}) \subset \operatorname{Ker}(y_{2k})$.
Since $R$ is right semihereditary, $\operatorname{Im}(y_{2k+2})$ is a projective
$R$-submodule of $\null^n\mkern-2mu R$, and, hence, $\operatorname{Ker}(y_{2k+2})$ is an
$R$-summand of $\null^n\mkern-2mu R$,
and, therefore, of $\operatorname{Ker}(y_{2k})$.
Thus, \mbox{$\operatorname{Ker}(y_{2k}) =  \operatorname{Ker}(y_{2k+2}) \oplus L_k$}
for some   $L_k \ne \{0\}$. Then,\vspace{-2mm}
$$\null^n\mkern-2mu R = \operatorname{Ker}(y_{0})
= \operatorname{Ker}(y_{2k+2}) \oplus (\textstyle\bigoplus\limits_{i=0}^k L_i).\vspace{-2mm}$$ Each term of the
directed system
  $(\null^n\mkern-2mu R/(\bigoplus_{i=0}^k L_i) : k\in \naturals)$\vspace{1mm} is then $R$-projective.
On taking the direct limit,
we see that $ \null^n\mkern-2mu R/(\bigoplus_{i=0}^\infty L_i)$ is finitely generated, not finitely presented,
and flat.  Hence, $R$~is not right Sahaev.

\eqref{it:jon2} As \mbox{$\text{\normalfont{w.gl.dim\,}}R \le 1$}, every finitely generated
right ideal of $R$ is flat, and
hence  projective, since $R$ is right Sahaev.   Thus, $R$ is right semihereditary.   By
the contrapositive of \eqref{it:jon1}, $R$ is left Sahaev.  Hence, $R$ is left semihereditary.
By Theorem~\ref{thm:BC}  \eqref{it:sh}$\Rightarrow$\eqref{it:left2},
every finitely generated, torsionless, right $R$-module is flat, and
hence  projective, since $R$ is right Sahaev.
\end{proof}

\begin{corollary}\label{cor:JondrupMain} The following hold.
\vspace{-2mm}
\begin{enumerate}[{\rm (i)}]
\setlength\itemsep{-2pt}
\item\label{it:J1} If  \,$\text{\normalfont{w.gl.dim\,}}R \le 1$ $($e.g.\,\,$R$ is left or right
$($\hskip-.5mm semi$)$\hskip-.5mm hereditary$)$ and $R$~is  left or right Sahaev
$($e.g.\,\,$R$ is left or right Noetherian or $R$ embeds in a skew field\,$)$, then $R$ is
  $\Pi$-semi\-hereditary.

\item \label{it:J2}  If $R$ is a commutative principal ideal domain, or more generally,
a commutative B\'ezout domain, or, even more generally, a commutative Pr\"ufer domain, then
$R$ is $\Pi$-semihereditary. \hfill\qed
\end{enumerate}
\end{corollary}

\begin{proof} \eqref{it:J1} is Theorem~\ref{thm:jondrup}\eqref{it:jon2} and its left-right dual,
while \eqref{it:J2} gives some cases of~\eqref{it:J1} where $R$ is semihereditary and embeds in a field.
\end{proof}

The following is~\cite[Theorem~2B]{Jensen}, attributed to Cohn.

 \begin{theorem}\label{thm:ex4}  {\normalfont (Cohn, 1969)}   Semifirs are
left and right semihereditary,  left and right Sahaev rings.  Hence, semifirs are
 $\Pi$-semihereditary.
\end{theorem}

\begin{proof}[Proof {\normalfont (after J\o ndrup~\cite[Corollary 3]{Jondrup3})}]  It is clear that semifirs are
left and right semihereditary.  From the proof of Theorem~\ref{thm:jondrup}\eqref{it:jon1},  it may be seen that if
$R$ is right semihereditary and not left Sahaev, then, for some $n \in \naturals$,
$\null^n\mkern-2mu R$ is a direct sum of $n{+}1$ non-zero right $R$-submodules; this
cannot happen if $R$ is a semifir.
Hence semifirs are  left (and right) Sahaev  rings, and Corollary~\ref{cor:JondrupMain}\eqref{it:J1} applies.
\end{proof}

We now turn to the final topic of this section,
self-injective rings; for  background on
 injective modules, see~\cite[pp.\,65--68]{Rotman}.

At the beginning of the proof of~\cite[Theorem~2 (c)$\Rightarrow$(a)]{Goodearl72},
Goodearl notes the consequence of~\cite[Theorem~2.1 (a)$\Rightarrow$(b)]{Cateforis}
 that, over a right self-injective, von Neumann regular ring, every torsionless right module is
projective.  This seems to have been one of the very earliest
 occasions on which
the projectivity of finitely generated torsionless modules was considered.

 \begin{theorem} \label{thm:Cateforis} {\normalfont (Cateforis, 1969)}
If $R$ is right semihereditary and right self-injec\-tive, then
$R$ is   $\Pi$-semihereditary.
\end{theorem}

\begin{proof}
By definition, $R$ is $\Pi$-semihereditary if each finitely generated, torsionless,
right $R$-module $M$  is projective. Since $M$ is finitely generated, we
 may write \mbox{$M = F/K$} where $F$ is a finitely generated, free,
right $R$-module, and $K$ is a submodule  of $F$.
In particular, $F$ is  right $R$-injective.
Let $\mathcal{L}$ denote the set of submodules of $F$ whose intersection with $K$ is $\{0\}$.
Clearly, $\mathcal{L} \ne \emptyset$.
By Zorn's lemma,  $\mathcal{L}$~contains some $\subseteq$-maximal element $L$.
Then  $F \supseteq K{\oplus}\,L$.
We then have   injective  right $R$-module maps \mbox{$ K \to F/L$}, $k \mapsto k{+}L$, and $K \to F$, $k \mapsto k$.
Since  $F$ is right $R$-injective, the domain of the latter map can be transported along the former map to give
a right $R$-module map $\phi \colon F/L \to F$ such that, for all $k \in K$, $\phi(k{+}L) = k$.

Define $\phi' \colon F \to F$ by $f \mapsto \phi(f{+}L)$. Then $\operatorname{Ker}(\phi') \supseteq L$,
and, for all $k \in K$, $\phi'(k) = k$.  If  $\operatorname{Im}(\phi') = K$, then $F/K$ is projective,
as desired.
Thus, it suffices to assume that  $\operatorname{Im}(\phi') \supset K$, and obtain a contradiction.

Define $\phi''  \colon F \to M$ by \mbox{$f \mapsto  \phi'(f){+}K \in F/K = M$}.
Then there exists some
\mbox{$m \in \operatorname{Im}(\phi''){-} \{0\}$}.
Thus, $F \supset \operatorname{Ker}(\phi'') \supseteq K{\oplus}L$.

Since $M$ is torsionless,  we may compose $\phi''$ with a map from $M$ to $R$
which does not vanish on $m$, and thus
obtain a map \mbox{$\psi  \colon F \to R$} with  $F \supset \operatorname{Ker}(\psi) \supseteq  K{\oplus}L$.
Then $\operatorname{Im}(\psi)$ is a   finitely generated  right ideal of~$R$.
Since $R$ is right semihereditary, $\operatorname{Im}(\psi)$
  is  projective. Hence,
$F = \operatorname{Ker}(\psi) \oplus L'$ for some $L' \ne \{0\}$.
Now $F \supseteq K {\oplus}L {\oplus} L'$  and this contradicts the $\subseteq$-maximality of $L$ in $\mathcal{L}$, as desired.
\end{proof}

\begin{remarks}\label{rems:vnr}  Recall that if $R$ is von Neumann regular, then every
finitely presented right $R$-module is
projective by Theorem~\ref{thm:Vil}, hence every finitely generated right ideal of $R$
is a summand of $R$, and hence $R$ is right semihereditary.

Recall also that if $R$ is
right semihereditary and right self-injective, then every
finitely generated right ideal is projective, hence injective, hence a summand of $R$,
and $R$ is von Neumann regular.
\end{remarks}

 We shall be using~\cite[Corollary]{Osofsky}, which is as follows.

\begin{theorem}\label{thm:Osofsky}  {\normalfont (Osofsky, 1964)} If $R$ is right self-injective and
right hereditary, then $R$ is completely reducible.
\end{theorem}

\begin{proof}  Suppose not. We shall obtain a contradiction.
Let $\mathcal{U}$ denote the set of right ideals of $R$ that are not finitely generated.
By Remarks~\ref{rems:vnr}, $R$ is von Neumann regular; since $R$ is  not completely reducible,   $\mathcal{U} \ne \emptyset$.
 By Zorn's lemma, $\mathcal{U}$ contains some $\subseteq$-maximal element  $I$.

Since $R$ is right   hereditary,  $I$~is  right $R$-projective.
Since $R$ is von Neumann regular, results of Kaplansky show  that  $I$ is a direct sum of
infinitely many nonzero right ideals of $R$; in the case where $I$ is $\aleph_0$-generated,
this follows from~\cite[Proof of Lemma 1]{Kaplansky},
and in the case where $I$ is not $\aleph_0$-generated, it follows from~\cite[Theorem~1]{Kaplansky2}.
There then exists some decomposition $I = I' \oplus I''$ with $I'$, $I'' \in \mathcal{U}$.

We shall now show that there exists some right ideal $\overline{I'}$ of $R$ with the properties that
$I' \subseteq  \overline{I'}$, $\overline{I'} \not \in \mathcal{U}$, and $\overline{I'} \cap I'' = \{0\}$;
in Remarks~\ref{rems:alt}, we shall give an alternative argument which does not require
background on injective hulls.
Recall that Eckmann and Schopf~\cite{ES}
 showed that, since $R$ is right self-injective, $I'$ has a right $R$-injective hull $\overline{I'}$ in $R$
which is a maximal right $R$-essential extension of $I'$ in $R$;
see~\cite[proof of Theorem~3.30\,(iii)$\Rightarrow$(i)]{Rotman}.
Then $\overline{I'}$ is a
right $R$-summand of $R$, and, hence, $\overline{I'}$ is right $R$-cyclic.
Thus, $\overline{I'} \not \in \mathcal{U}$. Also, since $I' \cap I'' = \{0\}$,
it follows that  $\overline{I'} \cap I'' = \{0\}$, as claimed.

Since $\overline{I'} \not \in \mathcal{U}$, we have $I \ne \overline{I'}$  and $I \subset \overline{I'}$.
Since $I'' \in \mathcal{U}$,
it can be seen that \mbox{$\overline{I'} \oplus I'' \in \mathcal{U}$}.
Now \mbox{$I = I' \oplus I'' \subset \overline{I'} \oplus I''$}.  This contradicts
the maximality of $I$ in~$\mathcal{U}$ and completes the proof.
\end{proof}

\begin{remarks}\label{rems:alt}  Here is an alternative proof of the existence of an $\overline{I'}$.
Let \mbox{$M \coloneq R/I'$}, and consider the double-dual map
\mbox{$\phi \colon M \to M^{**}$} as in the proof of Theorem~\ref{thm:Y}.
There exists a right ideal $\overline{I'}$ of $R$ such that
\mbox{$\operatorname{Ker}(\phi) = \overline{I'}/I'$}, and then
\mbox{$\operatorname{Im}(\phi) \simeq R/\overline{I'}$}.
By Theorem~\ref{thm:Cateforis}, $R$ is right $\Pi$-coherent.
It is clear that $\operatorname{Im}(\phi)$
is a finitely generated, torsionless, right $R$-module, and, hence, $\operatorname{Im}(\phi)$  is finitely presented.
Thus, $\overline{I'}$  is finitely generated, and, hence,  \mbox{$\overline{I'} \not \in \mathcal{U}$}.
We shall now show that \mbox{$\overline{I'} \cap I'' = \{0\}$}.
Consider any $x \in I''$.  Then \mbox{$xR = eR$}   for some
 idempotent \mbox{$e\in R$}, since $R$ is von Neumann regular.
Since \mbox{$eR \subseteq I''$}, we see that \mbox{$eR \cap I' = \{0\}$}, and, hence, $eR$
embeds in \mbox{$R/I' = M$}.
Since $R$ is right self-injective, $eR$ is right $R$-injective.
 Thus, $eR$  becomes  a  summand of $M$ which, as $eR$ is projective, must then
 embed  in $M^{**}$.   We now see that \mbox{$eR \cap \overline{I'} = \{0\}$}.
It then follows that \mbox{$\overline{I'} \cap I'' = \{0\}$}, as claimed.
\end{remarks}

 The following combines~\cite[Theorem~2 (a)$\Rightarrow$(c)]{Goodearl72} and the commutative case of Theorem~\ref{thm:Osofsky}.

\begin{theorem}\label{thm:GoodO}     Let $R$ be a commutative, $\Pi$-coherent,
von Neumann regular ring.
\vspace{-2mm}
\begin{enumerate}[{\rm (i)}]
\setlength\itemsep{-2pt}
\item \label{it:GoodO1}  {\normalfont (Goodearl, 1972)}  $R$ is self-injective.
\item\label{it:GoodO2}  {\normalfont (Osofsky, 1964)}  If $R$ is   $\aleph_0$-Noe\-therian or, more generally,
 hereditary, then $R$ is  completely reducible.
\end{enumerate}
\end{theorem}

\begin{proof} \eqref{it:GoodO1}   (based loosely on Kobayashi \cite[Theorem~2]{Kobayashi})\,\,
By Baer's criterion~\cite[Theorem~3.20]{Rotman}, to show  that   $R$ is self-injective,
it suffices to show that, for each ideal $I$ of $R$,  each $R$-module map
 \mbox{$\phi \colon I \to R$} is given by multiplication by some element of~$R$.

Now
  $\pi_1 \colon   R^2 \to R$,
$ (x,y) \mapsto x$,  and $\psi \colon   R^2 \to   R^I$,
\mbox{$(x,y) \mapsto   (\phi(x{\cdot}i) -  y{\cdot}i : i \in I),$}
are $R$-module maps.
Then $\operatorname{Im}(\psi)$
is a finitely generated, torsionless $R$-module.
Since $R$ is $\Pi$-coherent,
$\operatorname{Ker}(\psi)$ is a finitely generated $R$-submodule of~$R^2$.
 Hence, $\pi_1(\operatorname{Ker}(\psi))$ is a finitely generated ideal of~$R$.
Since $R$ is von Neumann regular, there exists some idempotent $e \in R$ such that
 $\pi_1(\operatorname{Ker}(\psi)) =  R{\cdot}e$.  For each $x \in R$, we now see that
$x \in R{\cdot} e$ if and only if there exists some $y \in R$ such
that,  for each $ i \in I$,  $ \phi(x{\cdot}i) =  y{\cdot}i$, that is,  $ \phi(i{\cdot}x) =  i{\cdot}y$,
since $R$ is commutative.

We apply  the `if' part taking $x \in I$ and $y = \phi(x)$.
For each $i \in I$, we have \mbox{$ \phi(i{\cdot}x) = i {\cdot} \phi(x) =  i {\cdot} y $}.
Hence,  $x \in R{\cdot} e$.  Thus, $I \subseteq R{\cdot}e$ and $i  {\cdot} e = i$ for all $ i \in I$.

We now apply  the `only if' part taking $x = e$.  There exists some $y \in R$ such that,  for each $i \in I$,
$  \phi(i{\cdot}e) = i{\cdot}y $, that is,
$\phi(i) =  i{\cdot}y$, as desired.

 \eqref{it:GoodO2}  Recall that a right $\aleph_0$-Noetherian, von Neumann regular ring is
right hereditary, by~\cite[Lemma 1]{Kaplansky}.  Now  \eqref{it:GoodO2} follows from~\eqref{it:GoodO1}
and Theorem~\ref{thm:Osofsky}.
\end{proof}

\begin{remarks}
If  $\text{\normalfont{w.gl.dim\,}}R \ge 2$, then $R$ is not
$\Pi$-semihereditary.

If $\text{\normalfont{w.gl.dim\,}}R  =  1$ and $R$ is right  Sahaev, then
  $R$ is left Sahaev and $\Pi$-semihered\-i\-tary, by Theorem~\ref{thm:jondrup}\eqref{it:jon2}.
This class includes all the semifirs that are not skew fields, by Theorem~\ref{thm:ex4}.
It also includes all the Pr\"ufer domains that are not fields.

If $\text{\normalfont{w.gl.dim\,}}R = 0$ and $R$ is left or right self-injective,
then $R$ is $\Pi$-semihereditary, by Theorem~\ref{thm:Cateforis}.  This class
includes all the  right semihereditary, right self-injective rings.
It also includes all the commutative $\Pi$-semi\-hered\-itary rings which have
 \mbox{$\text{\normalfont{w.gl.dim\,}}R = 0$}, by Theorem~\ref{thm:GoodO}\eqref{it:GoodO1}.
It also includes all the completely reducible rings.

If $\text{\normalfont{w.gl.dim\,}}R = 0$ and $R$ is right  Sahaev,
then clearly every finitely generated right $R$-module is projective,
 and $R$ is then completely reducible.  Such rings are included in the previous  case.

 Thus, we have   two  disjoint  classes of examples of $\Pi$-semihereditary rings.
\end{remarks}

\section{If $R$ is $\Pi$-semihereditary, then $\alpha$ is injective}\label{sec:PID}

For the remainder of the article, we fix the following.

\begin{notation}\label{not:main2}  For each $n \in \naturals$, we  have a multiplication map
that will be denoted $(\prod\limits_{\naturals} R)^{\otimes n} \xrightarrow{\beta_n} \prod\limits_{\naturals^n} R$.

Whenever Hypotheses~\ref{hyp:main} hold, we have a natural factorization of $\alpha$ that will be denoted
  $A \amalg_R B \xrightarrow{\alpha'} R\series{a} \amalg_R R\series{b}
\xrightarrow{\alpha''} R\series{a,b}.$
\end{notation}

\begin{proposition}\label{prop:suff01} If {\,\normalfont Hypotheses~\ref{hyp:main}} hold and
  $A$, $B$, and $ \prod\limits_{\naturals} R $ \vspace{-2mm} are left and right $R$-flat, then
$A \amalg_R B \xrightarrow{\alpha'} R\series{a} \amalg_R R\series{b}$ is injective.
\end{proposition}

\begin{proof} Clearly, the $R$-bimodule map\vspace{-2mm} \mbox{$\prod\limits_{\naturals} R \to a R\series{a}$},
\mbox{$ (r_i : i \in \naturals ) \mapsto \sum_{i\ge 0} (r_i a^{i+1})$}, is bijective.
Similarly for
$b R\series{b}$.

  We see that $ a R\series{a}$ and $b R\series{b}$ are left and right
$R$-flat.

 We use Notation~\ref{not:2}.  Now,  $A = R \oplus \mathfrak{a}$ and $ B = R \oplus \mathfrak{b}$;
hence, $\mathfrak{a}$ and
$\mathfrak{b}$ are left and right $R$-flat.

Let $n \in \naturals$.
For $1 \le i \le n$, let $M_i$ denote one of
$a R\series{a}$, $\mathfrak{a}$, $b R\series{b}$, or $\mathfrak{b}$.
It follows by induction that the $R$-bimodule
$M_1 \otimes M_2 \otimes \cdots \otimes M_n$ is left and right $R$-flat.
It then follows by induction that each of
the natural maps
\begin{align*}
(\mathfrak{a} \otimes \mathfrak{b})^{\otimes n} &\to (a R\series{a} \otimes bR\series{b})^{\otimes n},
\\
(\mathfrak{b} \otimes \mathfrak{a})^{\otimes n} &\to ( bR\series{b} \otimes a R\series{a})^{\otimes n},
\\(\mathfrak{a} \otimes \mathfrak{b})^{\otimes n}\otimes \mathfrak{a}
& \to (a R\series{a} \otimes b R\series{b})^{\otimes n} \otimes a R\series{a},
\\
(\mathfrak{b} \otimes \mathfrak{a})^{\otimes n} \otimes \mathfrak{b}
&\to ( bR\series{b} \otimes a R\series{a})^{\otimes n} \otimes bR\series{b},
\end{align*}
is injective.
By using the standard $R$-bimodule decompositions of $A \amalg B$ and
\mbox{$R\series{a} \amalg R\series{b}$}, we see that, as an $R$-bimodule map,
$\alpha'$ is a direct sum of these injective maps, and hence is injective.
\end{proof}

\begin{proposition}\label{prop:suff02} Suppose that {\,\normalfont Hypotheses~\ref{hyp:main}} hold.
\vspace{-2mm}
\begin{enumerate}[{\rm (i)}]
\setlength\itemsep{-2pt}
\item $R\series{a} \amalg_R R\series{b}
\xrightarrow{\alpha''} R\series{a,b}$ is injective if and only if, for each $n \in \naturals$,
$(\prod\limits_{\naturals} R)^{\otimes n} \xrightarrow{\beta_n} \prod\limits_{\naturals^n} R$ is injective.
\item If $ R $ is left or right  coherent and
$\beta_2$ is injective, then $\alpha''$ is injective.
\item If $\prod\limits_{\naturals} R $ is left or right $R$-coherent, then $\alpha''$ is injective.
\end{enumerate}
\end{proposition}

\begin{proof}   (i)    We may assume that $A = R\series{a}$, $B = R\series{b}$ and $\alpha = \alpha''$.
We use Notation~\ref{not:2}. Let $n \in \naturals$.
Consider the map
$(\mathfrak{a}\otimes \mathfrak{b})^{\otimes n} \to (\mathfrak{a}\mathfrak{b})^n \subseteq R\series{a,b}$.
The $R$-bimodule isomorphisms $ \prod\limits_{\naturals} R \to \mathfrak{a}$ and
$ \prod\limits_{\naturals} R \to \mathfrak{b}$ give an embedding
\mbox{$(\prod\limits_{\naturals} R)^{\otimes (2n)} \to R\series{a} \amalg R\series{b}$} with image
$(\mathfrak{a}\otimes \mathfrak{b})^{\otimes n}$.
They also give an embedding
$ \prod\limits_{\naturals^{2n}} R \to R\series{a,b}$ with image $\widehat{ (\mathfrak{a}\mathfrak{b})^n}$,
the closure of $(\mathfrak{a}\mathfrak{b})^n $ in $ R\series{a,b}$.
Thus the map $(\mathfrak{a}\otimes \mathfrak{b})^{\otimes n} \to \widehat{ (\mathfrak{a}\mathfrak{b})^n}$ is
a copy of $\beta_{2n}$.
We now see that $\alpha$ is a direct sum of $R$-bimodule maps
that consist of one copy of $\beta_0$ and two copies of
$\beta_n$ for each $n \ge 1$.
In particular, $\alpha$ is injective if and only if all the $\beta_n$, $n \in \naturals$, are injective.

(ii)   By symmetry, we may assume that $ R $ is left  coherent.

We shall show by induction that, for each $n \in \naturals$, $\beta_n$ is injective.
It is clear that $\beta_0$ is injective.
Suppose that $n \ge 0$, and that $\beta_n$ is injective.
Notice that $\beta_{n+1}$ may be factored as
\vspace{-2mm}$$\textstyle (\prod\limits_{\naturals} R) \otimes ((\prod\limits_{\naturals} R)^{\otimes n})
\xrightarrow{(\prod\limits_{\naturals} R) \otimes \beta_n}
(\prod\limits_{\naturals} R) \otimes (\prod\limits_{\,\,\,\,\,\naturals^n} R)
\xrightarrow{\beta_2'} \prod\limits_{\,\,\,\,\,\,\,\naturals^{n+1}} R.$$

Since  $R$ is left coherent,
by Theorem~\ref{thm:Chase} \eqref{it:Chasea}$\Rightarrow$\eqref{it:Chaseb},
$\prod\limits_{\naturals} R$ is right $R$-flat, and, hence,
$(\prod\limits_{\naturals} R) \otimes \beta_n$ is injective.
Since
there exists a bijection between $\naturals^n$ and $\naturals$,
we see that $\beta_2'$ may be identified with $\beta_2$ and is then injective.
Hence, the composition $\beta_{n+1}$ is injective.
This completes the inductive argument.
By (i), $\alpha''$ is injective.

(iii)  By symmetry, we may assume that  $\prod\limits_{\naturals} R $ is left $R$-coherent.
By Theorem~\ref{thm:Good}~\eqref{it:Good11}$\Rightarrow$\eqref{it:Good12},
 $R$  is   left coherent and  $\beta_2$ is injective.
By (ii), $\alpha''$ is injective.
\end{proof}

\begin{corollary}\label{cor:suff0} With
 {\,\normalfont Hypotheses~\ref{hyp:main}},  the following hold.
\vspace{-2mm}
\begin{enumerate}[{\rm (i)}]
\setlength\itemsep{-2pt}
\item  If $\prod\limits_\naturals R$   is a left or right
semi\-hereditary $R$-module,
then $\alpha$~is injective.
\item  If $R$ is a $\Pi$-semihereditary ring,
then $\alpha$~is injective.
\item  If $R$ is   right semihereditary \mbox{$($e.\,g.\,\,$\text{\normalfont{w.gl.dim\,}}R = 0)$}
 and $R$ is   right self-injective,
then $\alpha$~is injective  \mbox{$($and $\text{\normalfont{w.gl.dim\,}}R = 0)$}.
\item  If  \,$\text{\normalfont{w.gl.dim\,}}R \le 1$   \mbox{$($e.\,g.\,\,$R$ is left or right semihereditary$)$}  and $R$ is left or  right Sahaev, then
$\alpha$ is injective
$($and $R$ is left and right semihereditary and left and right Sahaev$)$.
\end{enumerate}
\end{corollary}

\begin{proof} (i) By symmetry, we may assume that
$\prod\limits_\naturals R$ is left $R$-semihereditary.
By Proposition~\ref{prop:suff02}(iii), $\alpha''$ is injective.

Notice that
$\prod\limits_\naturals R$ is left $R$-coherent and that
every left $R$-submodule of $\prod\limits_\naturals R$ is flat.
In particular,
$R$ is left coherent and $\text{w.\,gl.\,dim\,}R \le 1$.
By Theorem~\ref{thm:Chase} \eqref{it:Chasea}$\Rightarrow$\eqref{it:Chaseb},
$\prod\limits_\naturals R$ is right $R$-flat, and then all its
right $R$-submodules are flat.
Hence, $\prod\limits_\naturals R$, $R\series{a}$, $R\series{b}$, $A$, and $B$, are left and right $R$-flat.
By Proposition~\ref{prop:suff01}, $\alpha'$ is
injective.

Hence, $\alpha$ is injective.

\noindent\hskip18pt(ii) follows from (i).

(iii) follows from (ii) and  Theorem~\ref{thm:Cateforis}.

\noindent\hskip16pt(iv) follows from (ii) and  Corollary~\ref{cor:JondrupMain}\eqref{it:J1}.
\end{proof}

In the next section, we shall  see results of Bergman and  Goodearl that show that
 $\alpha$ can fail to be injective whenever $R$ is a commutative ring such that  either
 $\text{w.\,gl.\,dim\,} R \ge 2$ or $R$ is $\aleph_0$-Noetherian, non-self-injective, and von Neumann regular.

\section{Non-injectivity phenomena}\label{sec:ex}

In this section, we  give examples contributed by \mbox{G.\,M.\,\,Bergman}
where Hypotheses~\ref{hyp:main} hold and $\alpha$ is not injective.
Throughout,  we use the factorization of $\alpha$ as
  $A \amalg_R B \xrightarrow{\alpha'} R\series{a} \amalg_R R\series{b}
\xrightarrow{\alpha''} R\series{a,b}.$

We begin with a general example
where   $R[a] \subseteq A \subseteq R[a,(1{+}\,a)^{-1}]$ and \mbox{$R[b] \subseteq B \subseteq R[b,(1{+}\,b)^{-1}]$},
but
\mbox{$\alpha'$} is not injective.

\begin{proposition}\label{prop:wd2} {\normalfont (Bergman)} Suppose that
 $R$ has two-sided ideals $I$, $J$ such that the multiplication map $I \otimes J \to IJ$ is not
injective, e.g.  \mbox{$R = \integers/p^2\integers$}, $p$ prime, or $R = \rationals[x,y]$.  Then, with  a suitable  choice of $A$ and $B$,
   {\,\normalfont Hypotheses~\ref{hyp:main}}   hold  and $\alpha'$, $\alpha$ are not injective.
\end{proposition}

\begin{proof} For    $R = \integers/p^2\integers$, $p$ prime, we may take  $I= J = pR$, where $p\otimes p \mapsto 0$.

For  $R = \rationals[x,y]$, we may take $I=J= xR+yR$, where $(x\otimes y){-}(y\otimes x) \mapsto 0$.

 Let $h \coloneq 1{+}\,a$ and
$H \coloneq \gen{h} \le \text{U}(R\series{a})$.
Then $R[h] = R[a]$, and we view \mbox{$ R[h] \subseteq RH \subseteq R\series{a}$}.
We have $R$-bimodule decompositions
$$\textstyle RH = \bigoplus\limits_{i\in \integers} h^iR = R[h] \oplus \bigoplus\limits_{i \le -1} (h^iR).$$
In $RH$, take $A$ to be the subring $R[h] \oplus \bigoplus\limits_{i \le -1} (h^iI)$.
Since $h^{-1} = 1{-}h^{-1}a$, $ R {+\,} aA = A$.

Let $k \coloneq 1{+}\,\,b$,
$K \coloneq \gen{k} \le \text{U}(R\series{b})$.
We view $R[b] = R[k] \subseteq RK \subseteq R\series{b}$, and we have an
$R$-bimodule decomposition
$$\textstyle RK = R[k] \oplus \bigoplus\limits_{i \le -1} (Rk^i).$$
In $RK$, take $B$ to be the subring $R[k] \oplus \bigoplus\limits_{i \le -1} (Jk^i)$.
Since $k^{-1} = 1{-}k^{-1}b$, $ R {+}\, bB = B$.

We use Notation~\ref{not:2}. Here,
$$\textstyle\mathfrak{a} = aA = aR[h] \oplus \bigoplus\limits_{i \le -1} (h^iaI), \qquad\text{and}\qquad
\mathfrak{b} = bB = bR[k] \oplus \bigoplus\limits_{i \le -1} (Jbk^i).$$
In particular, $h^{-1}aI$ is an $R$-bimodule direct summand of $\mathfrak{a}$,
and $Jbk^{-1} $ is an $R$-bimodule direct summand of $\mathfrak{b}$.
Hence $(h^{-1}aI)\otimes  (Jbk^{-1})$ may be viewed as an $R$-bimodule direct summand of
$\mathfrak{a} \otimes \mathfrak{b}$, and also of $A \amalg B$.
This summand
does not map injectively to $(h^{-1}aI) (Jbk^{-1})$ in $(RH) \amalg (RK)$.
It follows that $\alpha'$ is not injective in this case.  Hence, $\alpha$ is not injective.
\end{proof}

Recall that $\text{\normalfont w.gl.dim\,} R \ge 2$
if and only if there exists  a   right ideal $I$ of $R$ that is not right $R$-flat;
recall also that $I$ is not right $R$-flat if and only if  there exists a  left ideal $ J$ of $R$ such that the
multiplication map $I \otimes_R J \to IJ$ is not injective.  The next result then
follows from Proposition~\ref{prop:wd2}.

\begin{corollary}\label{cor:wd2}  Suppose that
 $R$ is commutative and  that $\text{\normalfont w.gl.dim\,} R \ge 2$.  Then
with  a suitable  choice of $A$ and $B$,
   {\,\normalfont Hypotheses~\ref{hyp:main}}   hold  and  $\alpha$ is not injective. \hfill\qed
\end{corollary}

\begin{remarks}   Suppose  that  $R$ is a commutative Sahaev ring and that Hypotheses~\ref{hyp:main}  hold.

 By Corollary~\ref{cor:wd2},  if $\alpha$ is injective for all choices of  $A$ and $B$, then
$\text{\normalfont w.gl.dim\,} R \le 1$.

Conversely, by Corollaries~\ref{cor:suff0} and~\ref{cor:JondrupMain}~\eqref{it:J1},
 if  $\text{\normalfont w.gl.dim\,} R \le 1$,
then $\alpha$ is   injective.

Thus, among the commutative Sahaev rings $R$,
the ones for which $\alpha$ is always injective are the ones with $\text{\normalfont w.gl.dim\,} R \le 1$, that is,
the semihereditary ones.

Hence, among the commutative domains $R$, the ones for which $\alpha$ is always
injective are the commutative Pr\"ufer domains.

Also, among the commutative Noetherian rings~$R$,
the ones for which $\alpha$ is always
injective are the hereditary ones.
\end{remarks}

We now give examples where Hypotheses~\ref{hyp:main} hold and $\beta_2$ is not injective,
and, hence, by Proposition~\ref{prop:suff02}(i),
\mbox{$\alpha''$} is not injective.
These are then examples where
$A = R\series{a}$ and $B = R\series{b}$ in Hypotheses~\ref{hyp:main}
and $\alpha$ is not injective.
We remark that if we had been working
in the setting of commutative rings of power series, then $\alpha''$ would have looked like $\beta_2$.

The following result gives a method for constructing
examples where $\beta_2$ is not injective.

\begin{proposition}\label{prop:tensor} {\normalfont (Bergman)} Let $Q$ be a field.
For $i = 1, 2$,
let $R_i$ be a commutative, augmented $Q$-ring whose augmentation ideal $I_i$ is not finitely generated,
and let $X_i$ be an infinite set such that $I_i$ is $\vert X_i\vert$-generated.
Let \mbox{$R \coloneq (R_1 \otimes_Q R_2)/(I_1\otimes_Q I_2)$}, a commutative $Q$-ring.
Then
\mbox{$\textstyle  \mu_{X_1,X_2} \colon (\null^{X_1}\mkern-2mu  R) \otimes_R ( R^{X_2})
\to \null^{X_1}\mkern-2mu R^{X_2}$}
is not injective.

In detail, for $i = 1,2$, any family $m_i \coloneq (r_{i,j_i} : j_i \in X_i)$
that generates~$I_i$ as ideal of~$R_i$ may be viewed as an element of
$ \prod\limits_{X_i} R $,\vspace{-2.5mm}
and then
\mbox{$m_1 \otimes m_2$}
is a non-zero element of   $\operatorname{Ker}(\mu_{X_1,X_2})$.
\end{proposition}

\begin{proof}
Notice that $R_1 \otimes_Q R_2 = (Q \oplus I_1) \otimes_Q (Q \oplus I_2) = Q \oplus I_1 \oplus I_2 \oplus (I_1 \otimes_Q I_2)$,
$R = Q \oplus I_1 \oplus I_2$, $R_1$ and $R_2$ are subrings of $R$,
$I_1$ and $I_2$ are ideals of $R$ with product zero,
and $R$ is a commutative augmented $Q$-ring with augmentation ideal
$I \coloneq I_1 \oplus I_2$.

Here, $\mu_{X_1,X_2}$ carries $m_1 \otimes m_2$ to
$(r_{1,j_1}r_{2,j_2} : (j_1 ,j_2 ) \in X_1 \times X_2)$, which is
zero, by our definition of $R$.
Thus, it remains to show that $m_1 \otimes m_2 \ne 0$ in
$(\null^{X_1}\mkern-2mu R) \otimes_R ( R^{X_2})$, and we shall do this
by finding a quotient $R$-module in which the image of $m_1 \otimes m_2 $ is not zero.

Let \mbox{$M_1 \coloneq \null^{X_1}\mkern-1mu (R/I_2)$},
a quotient $R$-module of $\null^{X_1}\mkern-2mu  R$ with trivial $I_2$-action.
Let \mbox{$M_2 \coloneq (R/I_1)^{X_2}$},
a quotient $R$-module of $R^{X_2}$ with trivial $I_1$-action.
For $i = 1,2$, set $\overline M_i \coloneq M_i/(M_iI) = M_i/(M_iI_i)$,
a quotient $R$-module of $\prod\limits_{X_i} R$ with trivial $I$-action,
and let $\overline m_i$ denote the image of $m_i$ in $\overline M_i$.
Then we have a map $$\textstyle(\null^{X_1}\mkern-2mu R) \otimes_R ( R^{X_2}) \to
\overline M_1 \otimes_R \overline M_2 = \overline M_1 \otimes_Q \overline M_2,$$ with
$m_1 \otimes_R m_2$ mapping to
\mbox{$\overline m_1 \otimes_Q \overline m_2$}.
It suffices to show that
\mbox{$\overline m_1 \otimes_Q \overline m_2 \ne 0$}, and,
since $Q$ is a field, it suffices to show that $\overline m_1 \ne 0$ and $\overline m_2 \ne 0$.
By symmetry it suffices to assume that $\overline m_1 = 0$ and obtain a contradiction, as follows.

 The natural bijection $R_1 \to R/I_2$
induces a bijection $\null^{X_1}\mkern-2mu  R_1 \to M_1$, and it carries
$(\null^{X_1}\mkern-2mu  R_1)I_1 $ to $M_1I_1$.
Since   $m_1 \in \null^{X_1}\mkern-2mu R_1$ and  $\overline m_1 = 0 $, we have $m_1 \in (\null^{X_1}\mkern-2mu R_1) I_1$.
Hence,
there is some expression \mbox{$m_1 = \sum_{k = 1}^{\ell} v_kr_k$} with each
$v_k \in \null^{X_1}\mkern-2mu R_1$ and each $r_k \in I_1$.
For each $j_1 \in X_1$, the $j_1$st coordinate of $m_1$ is $r_{1,j_1}$, and, hence, $r_{1,j_1}$ lies in the ideal of
$R_1$ generated by~$\{ r_k : 1 \le k \le \ell \}$, a finite subset of~$I_1$.
It follows that $I_1$ is a finitely generated ideal of $R_1$.
This is a contradiction, as desired.
\end{proof}

The following expands on the last two paragraphs of~\cite{Goodearl72}.

\begin{proposition}\label{prop:vnr}   If
 $R$ is  a commutative, $\aleph_0$-Noeth\-er\-ian,  von Neumann regular ring that is not completely reducible,
then the map
 \mbox{$\alpha''\colon R\series{a} \amalg R \series{b} \to R \series{a,b}$} is not injective.
\end{proposition}

\begin{proof}  By  Theorem~\ref{thm:GoodO}\eqref{it:GoodO2}, $R$ is not  $\Pi$-coherent.
By the contrapositive of
\mbox{Corollary~\ref{cor:Herbera} \eqref{it:HT3}$\Rightarrow$\eqref{it:HT1}},
$\beta_2$ is not injective.  By Proposition~\ref{prop:suff02}(i),
\mbox{$\alpha''$} is not injective.
\end{proof}

\begin{example}  If $Q$ is a
field and \mbox{$R = Q[e_j : j \in \naturals]/(e_k e_\ell - \delta_{k,\ell} e_k : k, \ell \in \naturals),$} then
 $\alpha''$ is not injective, since
 $R$ satisfies the hypotheses of Proposition~\ref{prop:vnr}; the
countable case was first seen in the last two paragraphs of~\cite{Goodearl72}.
Let us give an alternative proof   by showing that this $R$
 has the form considered in Proposition~\ref{prop:tensor} with $X_1$ and $X_2$ countable.
Partition $\naturals$ into two infinite subsets $X_1$ and $X_2$.
For $i = 1,2$, let \mbox{$m_i \coloneq \{e_j : j \in X_i\}$},
let $I_i$ denote the ideal of $R$ generated by~$m_i$,
and let $R_i$ denote the $Q$-subring of $R$ generated by $m_i$.
Then we may make the identification \mbox{$R = (R_1 \otimes_Q R_2)/(I_1\otimes_Q I_2)$}, and, by Proposition~\ref{prop:tensor},
\mbox{$m_1 \otimes m_2$}
represents a non-zero element of $\operatorname{Ker}(\mu_{X_1,X_2})$.
We may identify  $\mu_{X_1,X_2}$    with~$\beta_2$.
\end{example}

 Here is a  result similar to Proposition~\ref{prop:tensor}, which uses
a different, rather curious argument to prove that an appropriate element
of $\operatorname{Ker} \mu_{X,X}$ is not zero.

\begin{proposition}\label{prop:gmb2} {\normalfont (Bergman)} Let $Q$ be a non-zero commutative ring,
$X$ an infinite set, $I$ the free $Q$-module on $X$, $R $ the
commutative augmented  $Q$-ring $Q \oplus I$ with $I^2 = \{0\}$, and
$m$ the element $(x:x\in X)$ of $\prod\limits_{X} R$.\vspace{-2mm}
Then $m \otimes m$ represents a non-zero element of $\operatorname{Ker}(\mu_{X,X})$.
\end{proposition}

\begin{proof} It is clear that $m \otimes_R m$ is an element of $\operatorname{Ker}(\mu_{X,X})$, and
it remains to show that $m \otimes_R m \ne 0$ in $ (\null^{X}\mkern-2mu  R) \otimes_R (  R^X)$;
we shall do this
by finding a quotient $R$-module in which the image of $m \otimes_R m $ is not zero.

Define  $R$-bimodules
 $P \coloneq \textstyle   \null^{X}\mkern-2mu  (R \otimes_Q R )^X$ and
 $\overline P \coloneq (R/I) \otimes _R P \otimes_R (R/I)$, and make the identification $\overline P = P/(IP+PI)$.
We have  natural $R$-bimodule maps
\begin{align*}\textstyle(\null^{X}\mkern-2mu  R) \textstyle\otimes_Q (  R^X)
&\to\textstyle  \null^{X}\mkern-2mu  (R \otimes_Q R )^X = P \to \overline P.
\end{align*}
Since  $I$ acts trivially on $\overline P$ on the left and the right, the $R$-bimodule map
\mbox{$\textstyle(\null^{X}\mkern-2mu   R) \textstyle\otimes_Q ( R^X) \to \overline P$} is $R$-centralizing, and,
hence, factors through the universal $R$-centralizing quotient
$\textstyle(\null^{X}\mkern-2mu   R) \textstyle\otimes_R ( R^X)$ of the domain.
Under the resulting map $\textstyle(\null^{X}\mkern-2mu   R) \textstyle\otimes_R ( R^X) \to \overline P$,
our element $m \otimes_R m$ is mapped to $ p + IP +PI $
where $p \coloneq  (x \otimes_Q y : (x,y) \in X \times X) \in P$.  Thus to show that $m \otimes_R m \ne 0$,
 it suffices to show that
$p \not \in IP + PI$.

For every $x,y\in X$ and $t\in R\otimes_Q R$, we shall write
$c_{x\otimes y}(t)\in Q$ for the coefficient of $x\otimes y$ in
the expression for $t$ with respect to the   $Q$-basis  $(X\cup\{1\})\otimes(X\cup\{1\})$   of
$R\otimes_Q R$.

Consider any  \mbox{$q=(q_{x,y}:(x,y)\in X\times X) \in P$}.
For each $x_0 \in X$, it can be seen that \mbox{$x_0 P =  \textstyle   \null^{X}\mkern-2mu  (x_0 Q \otimes_Q  R)^X$}.
Hence, if $q \in x_0 P$, then,  for all \mbox{$(x,y) \in  X \times  X$},
$c_{x\otimes y}(q_{x,y}) = 0$ if $x \ne x_0$.  Similar statements hold for~\mbox{$Px_0$}.
If $q \in IP+PI$, then there exist
 finite subsets $X_0, Y_0\subseteq X$ such that $q \in \sum_{x_0 \in X_0}x_0P + \sum_{y_0 \in Y_0}Py_0$\vspace{.7mm}.
Since $X$ is infinite, there exist    $ x \in  X{-}X_0$ and $y \in X{-}Y_0$,
and then $c_{x\otimes y}(q_{x,y}) = 0$.  Since $Q \ne \{0\}$, $c_{x\otimes y}(q_{x,y}) \ne 1$.
Hence, $q \ne p$.
This shows that  $m \otimes_R m \ne 0$, as desired.
\end{proof}

In the previous two propositions,
we have created elements of $\operatorname{Ker}(\mu_{X_1,X_2})$  of the form $m_1 \otimes m_2$,
and this requires  the existence of zero-divisors in $R$.  By  using elements
of $\operatorname{Ker}(\mu_{X_1,X_2})$  of the form $m_1 \otimes m_2 - m_3 \otimes m_4$,
we can avoid having zero-divisors in $R$, as follows.

\begin{proposition} {\normalfont (Bergman)} Let $Q$ be a field, and let $R$ be the  subring of $Q[x,y]$
generated by  $\{x^i y : i\in\naturals\}$.  In particular, $R$ is a commutative  domain.
In $\prod\limits_{\naturals} R$, let\vspace{-3mm}
$$m_1 \coloneq m_2 \coloneq (x^{2i+1}y : i \in \naturals), \quad m_3 \coloneq (x^{2i}y:i\in \naturals),
\quad  m_4 \coloneq (x^{2i+2}y:i\in \naturals).$$
Then
$m_1 \otimes m_2 - m_3 \otimes m_4$ lies in $\operatorname{Ker}(\beta_2)-\{0\}$.
\end{proposition}

\begin{proof} For $i, j \in \naturals$, both $m_1 \otimes m_2$ and   $m_3 \otimes m_4$
have  $x^{2i+2j+2}y^2$ as their $(i,j)$th component.
Thus $m_1 \otimes m_2 - m_3 \otimes m_4\in \operatorname{Ker}(\beta_2)$.

Let $I$ denote the ideal of $R$ generated by $\{x^{2i}y, x^{2i} y^{2} : i  \in \naturals\} $,
and let $R' \coloneq R/I$.  Then $R'$ is the commutative augmented $Q$-ring with augmentation ideal of
square zero and $Q$-basis  $X \coloneq \{x^{2i+1} y : i \in\naturals\}$.
In $\prod\limits_{\naturals} R'$, let $m \coloneq (x^{2i+1}y : i \in \naturals)$.
By Proposition~\ref{prop:gmb2}, in  $(\prod\limits_{\naturals} R')^{\otimes 2}$, $m \otimes m \ne 0$.
The natural map $(\prod\limits_{\naturals} R)^{\otimes 2} \to (\prod\limits_{\naturals} R')^{\otimes 2}$
carries $m_1 \otimes m_2 - m_3 \otimes m_4$ to $(m \otimes m) - (0\otimes 0) \ne 0$.
Hence, $m_1 \otimes m_2 - m_3 \otimes m_4 \ne 0$.
\end{proof}

\bigskip

\noindent{\textbf{\Large{Acknowledgments}}}

\medskip
\footnotesize The first-named author was partially supported by
DGI MICIIN MTM2011-28992-C02-01, and by the Comissionat
per Universitats i Recerca de la Generalitat de Catalunya.

The second-named author was
partially supported by
Spain's Ministerio de Ciencia e Innovaci\'on
through Project MTM2011-25955.

We are very grateful to George Bergman for generously contributing the examples of Section~\ref{sec:ex} and
suggesting many improvements of  various earlier versions.

We are also grateful to Dolors Herbera  for providing much useful information for Section~\ref{sec:back}.

\bibliographystyle{amsplain}

\begin{thebibliography}{24}
\vskip-0.4cm \null
\bibitem{AD}
Pere Ara and Warren Dicks,
\newblock{\em Universal localizations embedded in power-series rings},
\newblock Forum Math.\ \textbf{19} (2007), 365--378.
\vskip-0.4cm \null
\bibitem{Bass}
 Hyman Bass,
\newblock{\em Finitistic dimension and a homological generalization of semi-primary rings\/},
\newblock Trans.\,Amer.\,Math.\,Soc.\,\,\textbf{95} (1960), 466--488.
\vskip-0.4cm \null
\bibitem{Bergman90}
George M.\,\,Bergman,
\newblock{\em Ordering coproducts of groups and semigroups},
\newblock J.\ Algebra \textbf{133} (1990), 313--339.
\vskip-0.4cm \null
\bibitem{Cateforis}
Vasily C.\,\,Cateforis,
\newblock{\em On regular self-injective rings},
\newblock Pacific J.\ Math.\ \textbf{30} (1969), 39--45.
\vskip-0.4cm \null
\bibitem{Chase}
Stephen U.\,\,Chase,
\newblock{\em Direct products of modules\/},
Trans.\,Amer.\,Math.\,Soc.\,\,\textbf{97} (1960), 457--473.
\vskip-0.4cm \null
\bibitem{Cohn59}
P.\,M.\,\,Cohn,
\newblock{\em On the free product of associative rings\/},
\newblock Math.\,Z.\,\,\textbf{71} 1959, 380--398.
\vskip-0.4cm \null
\bibitem{Cohn64}
P.\,M.\,\,Cohn,
\newblock{\em Free ideal rings\/},
\newblock J.\ Algebra \textbf{1} (1964), 47--69.
\vskip-0.4cm \null
\bibitem{Cohn85}
P.\,M.\,\,Cohn,
\newblock{\em Free rings and their relations. Second edition\/},
LMS Monographs \textbf{19}, Academic Press,   London, 1985. xxii+588 pp.
\vskip-0.4cm \null
\bibitem{Cox}
S.\,H.\,\,Cox and R.\,L.\,\,Pendleton,
\newblock{\em Rings for which certain flat modules are projective\/},
Trans.\,Amer.\,Math.\,Soc.\,\,\textbf{150} (1970), 139--156.
\vskip-0.4cm \null
\bibitem{ES}
B.\,Eckmann and A.\,Schopf,
\newblock{\em \"Uber injektive Moduln\/},
Arch.\,Math.\,\,\textbf{4} (1953), 75--78.
\vskip-0.4cm \null
\bibitem{Fox53}
Ralph H.\,\,Fox,
\newblock{\em Free differential calculus.\,I  Derivation in the free group ring},
\newblock Ann.\ of Math.\ (2) \textbf{57} (1953), 547--560.
\vskip-0.4cm \null
\bibitem{Goodearl72}
K.\,R.\,\,Goodearl,
\newblock{\em Distributing tensor product over direct product\/},
\newblock Pacific J.\ Math.\ \textbf{43} (1972), 107--110.
\vskip-0.4cm \null
\bibitem{Herbera}
Dolors Herbera and Jan Trlifaj,
\newblock{\em Almost free modules and Mittag-Leffler conditions\/}, Adv.\,Math.\ \textbf{229}
(2012), 3436--3467.
\vskip-0.4cm\null
\bibitem{Jensen}
C.\,U.\,\,Jensen,
\newblock{\em Some cardinality questions for flat modules and coherence\/},
\newblock J.\, Algebra \textbf{12} (1969), 231--241.
\vskip-0.4cm \null
\bibitem{Jondrup3}
S.\,\,J\o ndrup,
\newblock{\em On finitely generated flat modules\,\,III\/},
Math.\ Scand.\ \textbf{29} (1971), 206--210.
\vskip-0.4cm \null
\bibitem{Jones}
Marsha Finkel Jones,
\newblock{\em Flatness and $f$-projectivity of torsion-free modules and injective modules\/}, pp.\,94--116 in:
\newblock {\em Advances in noncommutative ring theory.
\!Proc.\,Twelfth George H. Hudson Symp., SUNY, Plattsburgh NY, 1981},
(ed. Patrick J.\,Fleury), Lecture Notes in Math.\, \textbf{951}, Springer-Verlag, Berlin, 1982. iii+142 pp.
\vskip-0.4cm \null
\bibitem{Kaplansky2}
Irving Kaplansky,
\newblock{\em Projective modules\/},
Ann.\,\,of Math.\,\,\textbf{68} (1958), 372--377.
\vskip-0.4cm \null
\bibitem{Kaplansky}
Irving Kaplansky,
\newblock{\em On the dimension of modules and algebras,\,X   A right hereditary ring which is not left hereditary\/},
Nagoya Math.\,\,J.\,\,\textbf{13} (1958), 85--88.
\vskip-0.4cm \null
\bibitem{Kobayashi}
Shigeru Kobayashi,
\newblock{\em A note on regular self-injective rings\/},
Osaka J.\,\,Math.\,\,\textbf{21} (1984), 679--682.
\iffalse
\vskip-0.4cm \null
\bibitem{Lawrence}
John Lawrence,
\newblock{\em A countable self-injective ring is quasi-Frobenius\/},
Proc.\,\,Amer.\,\,Math.\,\,Soc.\,\,\textbf{65} (1977), 217--220.  Erratum:
Ibid \textbf{73} (1979), 140.
\vskip-0.4cm\null
\bibitem{Lam}
T.\,Y.\,\,Lam,
\newblock{\em Lectures on modules and rings\/},
\newblock Graduate Texts in Mathematics \textbf{189}, Springer-Verlag, New York, 1999. xxiv+557 pp.
\fi
\vskip-0.4cm\null
\bibitem{Lenzing}
Helmut Lenzing,
\newblock{\em Endlich pr\"asentierbare Moduln\/},
Arch.\,Math.\,\,\textbf{20} (1969), 262--266.
\vskip-0.4cm \null
\bibitem{Magnus35}
Wilhelm Magnus,
\newblock{\em Beziehungen zwischen Gruppen und Idealen in einem speziellen Ring\/},
\newblock Math.\, Ann.\,\,\textbf{111} (1935), 259--280.
\vskip-0.4cm \null
\bibitem{Osofsky}
B.\,L.\,\,Osofsky,
\newblock{\em Rings all of whose finitely generated modules are injective\/},
\newblock Pacific J.\, Math.\,\,\textbf{14} (1964), 645--650.
\iffalse
\vskip-0.4cm \null
\bibitem{Osofsky2}
B.\,L.\,\,Osofsky,
\newblock{\em Noninjective cyclic modules\/},
\newblock Proc.\,\,Amer.\,\,Math.\,\,Soc.\,\,\textbf{19} (1968), 1383--1384.
\fi
\vskip-0.4cm \null
\bibitem{Rotman}
Joseph J.\,\,Rotman,
\newblock{\em An introduction to homological algebra\/},
\newblock Pure and applied math\-ematics~\textbf{89}, Academic Press, New York, 1979. xi+376 pp.
\vskip-0.4cm \null
\bibitem{Sahaev}
I.\,I.\,\,Sahaev,
\newblock{\em  The projectivity of finitely generated flat modules\/},
Mat.\,\v{Z}. \textbf{6} (1965), 564--573.
\end{thebibliography}

\noindent\textsc{Departament de Matem\`atiques,
Universitat Aut\`onoma de Barcelona,
E-08193 Bellaterra (Barcelona), Spain}

\medskip

\noindent \emph{E-mail addresses:} \url{para@mat.uab.cat},\,\,  \url{dicks@mat.uab.cat}.

\noindent \emph{Home page:} \url{http://mat.uab.cat/~dicks/}

\end{document}